\pdfoutput=1
\documentclass[a4paper,11pt]{amsart}

\usepackage{amsmath,amscd,amsfonts,amssymb,latexsym}
\usepackage{pgf}
\usepackage[width=5.64in,height=7.91in,centering]{geometry}

\newtheorem{theorem}{Theorem}[section]
\newtheorem{proposition}[theorem]{Proposition}
\newtheorem{lemma}[theorem]{Lemma}
\newtheorem{corollary}[theorem]{Corollary}
\newtheorem{remark}[theorem]{Remark}
\newtheorem{definition}[theorem]{Definition}
\newtheorem{example}[theorem]{Example}

\def\Hom{{\rm Hom}}
\def\C{\mathbb{C}}
\def\Q{\mathbb{Q}}
\def\T{\mathbb{T}}
\def\O{{\mathcal O}}
\def\D{{\mathcal D}}

\def\Z{\mathbb{Z}}
\def\P{\mathbb{P}}
\def\m{\mathfrak{m}}
\def\qed{\hfill$\Box$\s}
\def\s{\vskip10pt}
\def\codim{{\rm codim}}
\def\dim{{\rm dim}}

\title[Equivariant Hirzebruch class for singular varieties]
{ Equivariant Hirzebruch class \\for singular varieties}
\author{Andrzej Weber}\thanks{Supported by NCN grant 2013/08/A/ST1/00804
}
 \address{Department of Mathematics of Warsaw University\\
 Banacha 2, 02-097 Warszawa, Poland} \email{aweber@mimuw.edu.pl}

\begin{document}

\begin{abstract}  We develop an equivariant version of the
Hirzebruch class for singular varieties. When the group acting is a torus we apply
Localization Theorem of Atiyah-Bott and Berline-Vergne. The localized Hirzebruch class is an
invariant of a singularity germ.
 The singularities of toric varieties and Schubert varieties are of
special interest. We prove certain positivity results for simplicial toric varieties.
The positivity for Schubert varieties is illustrated by many examples, but it remains mysterious.
\end{abstract}


\maketitle

The main goal of the paper is to show how a theory of   global
invariants can be applied to study local objects equipped with an
action of a large group of symmetries.
 The theory of global invariants
we are going to discuss is the theory of characteristic classes,
more precisely the Hirzebruch class and $\chi_y$-genus. By \cite{Yo, BSY} the Hirzebruch class admits a
generalization for singular varieties.
 Let $X$ be an  algebraic variety in a compact complex algebraic manifold $M$. Suppose that $X$ is preserved by a torus $\T$ acting on $M$. For simplicity assume that the fixed point set $M^\T$ is discrete.
Then by
Localization Theorem of Atiyah-Bott and Berline-Vergne
the $\chi_y$-genus of $X$ can be written as a sum of contributions coming from fixed points. The contribution of a fixed point $p\in M^\T$ is equal to the equivariant Hirzebruch class restricted to that point and divided by the Euler class of the tangent space at $p$
$$\chi_y(X)=\sum_{p\in M^\T}\frac{td^\T_y(X\hookrightarrow M)_{|p}}{eu(p)}\,.$$
The local contributions to $\chi_y$-genus are fairly computable and they are expressed by  polynomials in characters of the torus.
We will describe all the necessary components of the described construction, give various examples and we will discuss positivity property of the localized Hirzebruch class.

A relation with the Bia\l ynicki-Birula decomposition \cite{B-B} will be given in a subsequent paper \cite{Wbb}.
The reader can find useful to look at the article \cite{We2},
where an elementary and self-contained introduction to equivariant characteristic classes   is given.

\tableofcontents

\section{Localization formulae}
The $\chi_y$-genus  of a  smooth compact complex algebraic
variety $X$ was defined  by Hirzebruch \cite{Hi} as a formal combination of Euler
characteristics of the sheaves of differential forms
 \begin{equation}\chi_y(X):=\sum_{p=0}^{\dim X} \chi(X,\Omega_X^p) y^p\,.\label{chidef}\end{equation}
 The
Hirzebruch-Riemann-Roch theorem allows to express that numerical
invariant as an integral
$$\chi_y(X)=\int_Xtd_y(X) =\int_X td(TX)\cdot ch(\Lambda_y(T^*X))\,.$$
The following multiplicative characteristic classes appear in the
formula
\begin{itemize}
\item  the Todd class $td(TX)$ is the multiplicative characteristic class associated to the power series $td(x)=\frac x{1-e^{-x}}$

\item $ch(\Lambda_y TX^*)$ is the composition of the total $\lambda$-operation
    $$\Lambda_y(-)=\sum_{p=0}^{\dim X}\Lambda^p(-) y^p$$
    and the Chern character. This way a multiplicative characteristic class is obtained. It is associated to the power series
     $\varphi(x)=1+y\,e^{-x}\,.$
\end{itemize}

The equivariant setup demands some explanation but  the applied
tools are basically  standard. Suppose $\T=(\C^*)^r$ acts on a
manifold $X$ with finite number of fixed points. The tangent bundle of $X$ is an equivariant $\T$-bundle.
The equivariant
cohomology $H^*_\T(X)$ constructed by Borel is a suitable graded ring,
where the equivariant characteristic classes live. In particular we consider the equivariant Hirzebruch class $td^\T_y(X)$. The class $td^\T_y(X)$ may be nontrivial in arbitrary high gradation, therefore by $H^*_\T(X)$ we understand not the direct sum but the product $\prod_{k=0}^\infty H^k_\T(X)$. The equivariant cohomology has
more structure than ordinary cohomology, since it is a module over
$H^*_\T(pt)$, which is the symmetric algebra spanned by characters
$\Hom(\T,\C^*)$. Our main tool is the Localization Theorem for
torus action. Essentially it is due to Borel \cite[Ch.XII
\S6]{Bo}, but here is the formulation by Quillen:
\begin{theorem}[\cite{Qu}, Theorem 4.4] The restriction map $$H^*_\T(X)\to H^*_\T(X^\T)$$\label{locH}
is an isomorphism after localization in the multiplicative system generated by nontrivial characters.\label{quloc} \end{theorem}
 Atiyah-Bott \cite[Formula 3.8]{AB} or Berline-Vergne \cite{BrVe}
Localization Formula (see Theorem \ref{locfo} and the following Integration Formula) allows to express integral of equivariant
forms in terms of some data concentrated at the fixed points. We apply the localization formula
to the
equivariant Hirzebruch class $td^\T_y(X)$. By rigidity (see \cite{Mu}) we obtain just the class
of gradation zero, the $\chi_y$-genus
\begin{theorem}$$\chi_y(X)=\sum_{p\in X^\T}\frac1{eu(p)}{td^\T_y(X)_{|p}}\,.$$\label{loc}\end{theorem}
Here the Euler class   $eu(p)=\prod w_i$ is the product of weights
$w_i\in \Hom(\T,\C^*)=H^2_\T(pt)$ appearing in the tangent
representation $T_pX$. Note that the  $\chi_y$-genus
can be written in the form
$$\chi_y(X)=\sum_{p\in X^\T}\frac1{eu(p)}td^\T(X)\cdot  ch^\T(\Lambda_y(T^*X))_{|p}\,.$$
The local contribution of each fixed point
$ch^\T(\Lambda_y(T^*X))_{|p}$ is multiplied by
\begin{equation}\frac1{eu(p)}td(X)^\T_{|p}=\prod \frac1{w_i}\cdot \prod
\frac{w_i}{1-e^{-w_i}}=\prod \frac1{1-e^{-w_i}}\,.\label{locconsm}\end{equation} This
multiplier is equal to the inverse of the Chern character of the
class of the skyscraper sheaf $\O_{\{p\}}$.
 In fact we can localize before applying Chern
character and work in the topolo\-gical equivariant $K$-theory defined by Segal \cite{Se}.
The  localization in $K$-theory has a parallel shape comparing with Theorem \ref{quloc}
\begin{theorem}[\cite{Se}, Proposition 4.1]The restriction map $$K_\T(X)\to K_\T(X^\T)$$\label{locK}
is an isomorphism after localization in the multiplicative system\footnote{The original formulation of the localization theorem is stronger. Here we apply it for the localization in the dimension ideal. For an arbitrary prime ideal $\frak p\subset K_\T(pt)=Rep(\T)$ there is a maximal group $H\subset \T$, such that $K_\T(X)_{\frak p}\stackrel{\sim}\to K_\T(X^H)_{\frak p}$. For a smaller ideal $\frak p$ we have to take a smaller subgroup $H$.  }
 generated by $$T_{w}-1\in K_\T(pt)$$ for all nonzero characters $w\in Hom(\T,\C^*)$, where $T_{w}$ is the element of $K_\T(pt)$ associated to the one dimensional representation with character $w$.
\end{theorem}
If $X$ is a compact algebraic variety, then it is enough to localize in the system   generated by $$T_{w}-1\in K_\T(pt)$$ for all nonzero characters $w\in Hom(\T,\C^*)$ appearing in the tangent representations at the fixed points.

The equality below is the $K$-theoretic counterpart of the integration Theorem \ref{loc}:
\begin{equation}\chi(X,E)=\sum_{p\in X^T}\frac 1{[\O_p]}[E]_{|p}\end{equation}
for any $E\in K_\T(X)$. In particular
\begin{equation}\chi_y(X)=\sum_{p\in X^T}\frac 1{[\O_p]}[\Lambda_y
T^*X]_{|p}\,.\end{equation}
 Note that from the numerical point of view there is no need to worry in which groups
(equivariant $K$-theory or cohomology) the computation is done. The passage from cohomology to $K$-theory
is reflected by the choice of a new variables, the Chern characters of the line bundles.

A version of the localization theorem in $K$-theory is discussed
in \cite[Corollaire 6.12]{Gr}. The formula is also valid for singular spaces \cite{Ba}. Baum computes holomorphic Lefschetz number instead of $\chi_y$-genus, but essentially his argument is the same.
The holomorphic Lefschetz number $$\sum_{i=0}^{\dim X}(-1)^i Tr\big(f^*\in End(H^i(X;\O_X))\big)$$ of a finite order automorphism $f:X\to X$  is expressed
by the sum of local contributions coming from the fixed points. The contribution at a fixed point $p$
is obtained from the power series
 $$\sum_{k=0}^\infty Tr\big(f^*\in End(\m_p^k/\m_p^{k+1})\big)s^k\,.$$
 Here $\m_p$ is the maximal ideal at $p$.
 This series turns out to be a rational function with regular value at $s=1$. The specialization gives the local contribution to the global Lefschetz number.
  If the point $p$ is smooth the contribution at $p$ is equal to
$$\prod\frac1{1-\lambda_i},$$ where the product is taken over the eigenvalues of $f$ acting on the
tangent space $T_pX$. This is exactly the shape of the formula (\ref{locconsm}).

\begin{example}\rm  Let $X=\P^1$ with the standard action of $\C^*$. The equivariant ring of a point is equal to the polynomial
algebra in one variable $t$, the generator of $\Hom(\C^*,\C^*)$.
The weight of the tangent representation at 0 is equal to $w_0=t$
and the weight of the representation at infinity is equal to $w_\infty=-t$. We
find that
 \begin{align*} \chi_y(\P^1)&=\frac1{w_0}\frac{w_0}{1-e^{-w_0}}\left(1+y\,e^{-w_0}\right)+
\frac1{w_\infty}\frac{w_\infty}{1-e^{-w_\infty}}\left(1+y\,e^{-w_\infty}\right)\\
&=\frac1{1-e^{-t}}\left(1+y\,e^{-t}\right)+
\frac1{1-e^{t}}\left(1+y\,e^{t}\right)\,.\end{align*} Let us
introduce a
 new variable $e^{-t}=T$. It  can also be considered  as an element of $K_\T(pt)$ and in fact the calculi can be thought to be performed in the localized $K$-theory.
  Now we have
$$\chi_y(\P^1)=\frac{1+y\,T}{1-T}+\frac{1+y\,T^{-1}}{1-T^{-1}}=  1-y$$
Let us remark about the formal setup where the above computations
were done. We have started from the equivariant cohomology of a
point $H^*_\T(pt)$ that is $\Q[t]$ completed in $t$. We are forced to complete since we consider the
infinite power series $td(\pm t)=\frac t{1-e^{\pm t}}$ and Chern character $e^{\pm t}$. Equally
well the computation might have been done in the equivariant $K$-theory of a
point $K_\T(pt)=\Z[T,T^{-1}]$. Formally it is a different object,
but from the computational point of view it makes no difference.
The $K$-theory completed in the dimension ideal $I$ (which is
generated by $T-1$) and tensored with $\Q$ is isomorphic to
$\Q[[t]]$ via the map $T\mapsto e^{-t}$.
\end{example}

The same remarks are
valid for tori of bigger dimensions: $\T=(\C^*)^r$. The Chern character is an isomorphism
$$ch^\T:R(\T)^{^\wedge}_I\simeq K_\T(pt)^{^\wedge}_I\otimes \Q\to H^*_\T(pt)^{^\wedge}_{\frak m}\simeq \Q[[t_1,t_2,\dots,t_r]]\,,$$
where $R(\T)$ is the representation ring of $\T$, the dimension ideal is generated
 by the elements $T_w-1$ for the nontrivial characters $w$.
  (According to our convention $H^*_\T(pt)=\prod_{k=0}^\infty H^k_\T(pt)$
  is already completed in the maximal ideal $\m=H^{>0}_\T(pt)$.)
 The isomorphism $ch^\T$ is the composition of the completion isomorphism
 $K_\T(pt)^{^\wedge}_I\simeq K(BT)$ (see \cite{AS} for a broader context) with nonequivariant Chern
 character $K(BT)\otimes \Q\to H^*(BT)=H^*_\T(pt)$.

\section{Localizing $\chi_y$-genus of singular varieties}
Singular spaces enter into the story due to a theorem of
Yokura and later by Brasselet-Sch\"urmann-Yokura \cite{Yo,BSY} who proved that the
Hirzebruch class $td_y$ admits a generalization to singular
algebraic varieties. For any map from $X\to M$, both possibly singular algebraic varieties, we
associate a homology class with closed supports (Borel-Moore homology)
 $$td_y(f:X\to M)\in H_*^{BM}(M)[y].$$
We are interested mainly in inclusions of singular varieties into
smooth ambient spaces, therefore due to Poincar\'e duality we can
identify the target group with cohomology. The
Brasselet-Sch\"urmann-Yokura class satisfies
\begin{itemize}
\item If $X$ is smooth and the map is proper then $$td_y(f:X\to M)=f_*td_y(X)$$

\item If $X=Z\sqcup U$ with $Z$ closed then $$td_y(f:X\to M)=td_y(f_{|Z}:Z\to M)+td_y(f_{|U}:U\to M)$$
\end{itemize}
The equivariant version can be developed  as well, see the details in \S\ref{eqho}. It is of special interest when the acting group is a torus $\T$. If $M$ is
singular the target group, the equivariant homology, is a bit
nonstandard object and eventually we perform all the
computations assuming that $M$ is smooth. Since the equivariant
cohomology of a point is nontrivial in positive gradations there is a
space for nontrivial invariants of germs of singular varieties. By the Atiyah-Bott or Berline-Vergne
localization theorem we have
$$\chi_y(X)=\int_M td^\T_y(X\hookrightarrow M)=
\sum_{p\in M^\T}\frac1{eu(p)}{td^\T_y(X\hookrightarrow M)}_{|p}\,.$$
 The summand
$\frac1{eu(p)}{td^\T_y(X\hookrightarrow M)}_{|p}$ is considered as the local contribution to $\chi_y$-genus of the point $p\in X^T\subset M^T$.
 If $p\in X^\T$ is a smooth point, then
$$\frac1{eu(p)}{td_y^\T(X\hookrightarrow M)_{|p}}=\prod
\frac{1+y\,e^{-w_i}}{1-e^{-w_i}}\,,$$ where $w_i$ are the weights
of the tangent representation $T_pX$ and Euler class is
computed with respect to weights of $T_pM$. The question remains
what are the contributions of  singular points? These local genera
are analytic invariants of the singularity germs. We would like to
understand what properties of a singular point are reflected by
this invariant. Assume that $M=\C^n$. The localized Hirzebruch class is  of
the form
$$td^\T(\C^n)\cdot
 polynomial\;in\; e^{-w_i}\;and\;y\,.$$
The polynomial in the formula is the equivariant Chern character of
$f_{\bullet*}\Omega^*_{X_\bullet}$, where $f_\bullet:X_\bullet\to X$ is a smooth
hypercovering of $X$. As a matter of a fact according to
\cite[Theorem 2.1]{BSY} the Hirzebruch class factors through
$K$-theory of coherent sheaves on $M$
 $$\begin{matrix}&_{td_y(-)}\\
 K(Var/M)&\longrightarrow& H_*^{BM}(M)[y]\\
 \hfill_{mC_*}\searrow&&\nearrow_{td(M)ch(-)}\\
 &G(M)\otimes \Z[y]\,.
 \end{matrix}$$
 Here $G(M)$ stands for the Grothendieck group of coherent sheaves on $M$.
The same construction is valid for equivariant Hirzebruch class.  The class  $$mC_*(X\hookrightarrow
M)_{|p}=\left[Rf_{\bullet*}\Omega^*_{X_\bullet}\right]_{|p}$$ in
the equivariant Grothendieck group $G_{\T}(pt)=K_\T(pt)$ is the main
protagonist of this paper. Equally well we can talk about its Chern character in equivariant (co)homology.

Our goal  is to clarify the definition, provide
examples and formulate  some statements for special classes of
singularities. Taking the opportunity in \S\ref{divisor} we relate the Aluffi construction of Chern-Schwartz-MacPherson classes \cite{Alog} with the direct definition of the Hirzebruch class for simple normal crossing divisor complements.
In \S\ref{conical} we compute the
Hirzebruch class for the affine cone of a variety contained in
$\P^n$ and give many other examples. The  Schubert varieties are of special interest, but
for now in \S\ref{dod} we will give just a bunch of computations.

 Setting $y=0$ we arrive to the question what is a relation between $mC_0(id_X)=ch(\O_{X_\bullet})$
with $ch(\O_X)$ for a closed variety $X$. How far
$td_y^\T(X\hookrightarrow M)$ is far from $td^\T(M)ch^T(\O_X)$?
One can treat the difference as an analogue of the Milnor class (\cite{Pa})
which is the difference between Chern-Schwartz-MacPherson class
and the expected Fulton-Johnson class \cite{FJ}. The equality
$$td_0^\T(X\hookrightarrow M)=td^\T(M)ch^T(\O_X)$$ holds for a class
of varieties with rational singularities or more general, almost
by the definition, for Du Bois singularities, see \cite[Example 3.2]{BSY}. Then $td_0^\T(id_X)$ agrees with
Baum-Fulton-MacPherson class. In particular we have equality for
\begin{itemize}
  \item normal crossing divisors,
  \item Schubert varieties,
  \item toric singularities,
  \item cone  hypersurfaces in $\C^n$ of degree $d\leq n$.
\end{itemize}
See \S\ref{bfm} for details. In general the difference between $td_y^\T$ and its expected value for complete intersections is  called the  Hirzebruch-Milnor class. Formulas for that class in terms of vanishing cycle sheaves and  in terms of a stratification are given in \cite{MSS}.

\section{Motivation for positivity}

{\it The question of positivity} is much more complicated, although
it seems to hold for a related class of singularities. At the moment we do not state
a reasonable general conjecture. The
positivity depends on numerical relations between higher derived
images of the sheaves $\Omega^p_{X_\bullet}$. There are natural
variables in which $ch^\T(\Omega^*_{X_\bullet})_{|p}$ is often a
polynomial with nonnegative coefficients.   Our choice of
variables depends on the ambient space $M$. The positivity
statement is motivated by an easy fact following for example from
\cite{PW}.
\begin{theorem}\label{posy} Fix a splitting of the torus $\T=(\C^*)^r$.
Let $M=\C^n$ be a representation of $\T$ such that  all weights of
eigenspaces are nonnegative combinations of the basis characters.
Let $X$ be an invariant subvariety of $M$. Then the equivariant
fundamental class $[X]_{|0}\in H^*_\T(pt)$ is a polynomial in
the basis characters with nonnegative coefficients.\end{theorem}
The positivity can also be proved using degeneration of $X$ to an union of invariant linear subspaces.

 A similar statement holds for
$K$-theory, but one has to assume that the variety $X$ has at most
rational singularities.  The positivity in $K$-theory of compact
homogeneous spaces $G/P$ demands introducing signs, as it already
has appeared in \cite[Theorem 1]{Br} for the global nonequivariant
case. If $X$ is a subvariety
 in  a homogeneous space and it has at worst  rational singularities then
$$[\O_X]=\sum_\alpha (-1)^{\dim X-\dim Y_\alpha}c_\alpha [\O_{Y_\alpha}]\,,$$
where the sum is taken over the set of Schubert varieties and the
integers $c_\alpha$ are nonnegative. They are the dimensions
of certain cohomology groups. The proof essentially relays on
Kleiman transversality theorem and the fact that for rational singularities
$Tor^*{\O_{X}}(\O_{Y_1},\O_{Y_1})$
 is concentrated in one gradation. A similar formula holds in
equivariant $K$-theory, \cite[Theorem 4.2]{AGM}. One can get rid
of the sign alternation by multiplying  the fundamental classes by
$(-1)^{\codim X}$. We will formulate the positivity condition in
equivariant cohomology. Let us introduce variables
$$S_i=T_i-1=e^{-t_i}-1\in H_\T^*(pt)\,,$$ for $t_i\in H^2_\T(pt)$.
These variables have a
geometric origin: when we fix a geometric approximation
$B\T_m=(\P^m)^r$, then $S_i$ restricts to $-ch(\O_{H_i})$, where
$H_i=(\P^m)^{i-1}\times\P^{m-1}\times(\P^m)^{r-i}$ is the
coordinate hyperplane in $B\T_m=(\P^m)^r$. The $K$-theoretic
analogue of the Theorem \ref{posy} is the following statement:

\begin{theorem} With the notation and  assumptions of Theorem \ref{posy} suppose that $X$ has rational singularities.
Then the cohomology class $(-1)^{\codim X}ch^\T([\O_X])_{|0}\in
H^*_\T(pt)$ is a polynomial in $S_1,S_2,\dots S_r$
with nonnegative coefficients.
\end{theorem}

For a proof see \S\ref{udod}, Theorem \ref{posa}.
 We have noticed that if $X$ has mild
singularities then a similar property holds for full Hirzebruch
class. To start we examine normal crossing divisors and their
complements. One easily verifies that positivity holds in
variables $S_i$ and $d=-1-y$. We expect that the positivity is
preserved when $X\subset M$ has sufficiently good resolution, but
now we cannot formulate and prove a general result.
Apparently  we need a stronger condition than just rationality. Here we wish to give some examples
\begin{itemize}
 \item smooth
quadratic cone hypersurfaces in $\C^n$,
 \item simplicial
 toric singularities (Theorem \ref{torsimp}),
 \item du Val surface singularities (we omit explicit computations),
 \item some examples of Schubert cells.
\end{itemize}
  The equivariant
 approach has an advantage that we can separate global  phenomena
 from local and study only the local properties of the
 singularity. In addition, the results of computations are just
 polynomials, not some abstract classes in a huge unknown object.

The reader can find useful to look at the article \cite{We2},
where an elementary and self-contained introduction to equivariant characteristic classes   is given.

\section{Recollection of the Hirzebruch class}
The $\chi_y$-genus of a smooth and compact complex variety $X$ is
a formal combination of the Euler characteristics of the sheaves
of differential forms:
$$\chi_y(X)=\sum_{p=0}^{\dim X}\chi(X;\Omega^p_X)y^p=\sum_{p,q=0}^{\dim X} (-1)^q h^{p,q}(X)y^p\in\Z[y]\,,$$
see \cite
{Hi}. To shorten the notation we will write formally
$$\Omega^y_X=\Lambda_yT^*X=\sum_{y=0}^{\dim X}\Omega^p_X\,y^p\,.$$
By the Hirzebruch-Riemann-Roch theorem the $\chi_y$-genus is equal
to the integral
$$\int_X td(X)ch(\Omega_X^y)\,,$$
where $td(X)\in H^*(X)$ is the Todd class and $ch(-)$ is the Chern
character. (We always consider cohomology with rational
coefficients.) The characteristic class  $td(X)ch(\Omega_X^y)\in
H^*(X)[y]$ is called the Hirzebruch class and denoted by
$td_y(X)$. It is a multiplicative characteristic class associated
to the formal power series
$$td_y(x)=\frac{x(1+ye^{-x})}{1-e^{-x}}=x+(1+y)\left(\frac{x}{1-e^{-x}}-x\right)=x+(1+y)\left(td(x)-x\right)
 $$

$$=x+(1+y)\left(1 - \frac  x2 +  \frac {x^2}{12} -  \frac{x^4}{720} +  \frac{x^6}{30240} - \frac{
x^8}{1209600} + \frac{ x^{10}}{47900160}+\dots\right)$$

The normalized Hirzebruch (see \cite{Yo}) class is obtained from the power series
$$\widetilde{td}_y(x)=\frac{x(1+ye^{-(1+y)x})}{1-e^{-(1+y)x}}=\frac 1{1+y} td_y((1+y)x)=$$
$$=(1+x) - \frac { (1+y)x}2 +  \frac {(1+y)^2x^2}{12} -  \frac{(1+y)^4x^4}{720}
 +  \frac{(1+y)^6x^6}{30240} - \frac{
(1+y)^8x^8}{1209600} + 
\dots\,.$$
 For any vector bundle $E$ of rank $n$ over a $n$-dimensional variety we have  $$\int_M\widetilde{td}_y(E)=\int_M{td}_y(E)$$
 therefore
$$\int_X \widetilde{td}_y(TX)=\chi_y(X)\,.$$
In addition for $y=-1$ we recover the total Chern class of the
tangent bundle
$$\widetilde{td}_{-1}(TX)=c(TX)\,.$$
We will use unmodified Hirzebruch class having in mind, that
normalization is a matter of re-scaling homogeneous components of $td_y$ by a power of $(1+y)$. On the other hand according to  Thom the normalized power series are in bijection with multiplicative characteristic classes.
In fact for us it is irrelevant which version of $td_y$ we use. In all the computations for a cohomology class $t\in H^2(B\T)$ we use a variable $T=e^{-t}$ for unmodified version of the Hirzebruch genus. For the normalized version we would use $T=e^{-(y+1)t}$ and our formulas would remain not effected.

The Hirzebruch class was generalized for singular varieties in
\cite{BSY}. In fact it is well defined on the Grothendieck group
$K(Var/M)$ of varieties equipped with a map to a fixed variety
$M$. This means that for any map of possibly singular varieties
$f:X\to M$ the class $td_y(f:X\to M)$ is defined and
$$td_y(f:X\to M)=td_y(f_{|Y}:Y\to M)+td_y(f_{|U}:U\to M)$$
for any closed subvariety $Y\subset X$ and $U=X\setminus Y$. The
Hirzebruch class takes values in Borel-Moore homology
$H^{BM}_*(M)[y]$ (homology with closed supports) or in the Chow
group of $M$.  We obtain a transformation of functors
$$td_y:K(Var/-)\to H^{BM}_*(-)$$ defined on the category of
complex algebraic varieties with proper maps.  This means that if $\phi:M_1\to M_2$ is a proper map, then the following diagram is commutative:
\begin{equation}\begin{matrix}&_{td_y}\\
K(Var/M_1)&\longrightarrow& H^{BM}_*(M_1)\\
^{\phi_\circ}\downarrow^{\phantom{\phi_\circ}}&&\phantom{^{f_*}}\downarrow{^{\phi_*}}\\
K(Var/M_2)&\longrightarrow& H^{BM}_*(M_2)\,.\\
&^{td_y}\end{matrix}\label{funkt}\end{equation}
Here the map $\phi_\circ$ is just the composition sending $f:X\to M_1$ to $\phi\circ f: X\to M_2$.
If $X$ is smooth
and the map $f$ is proper, then
$$td_y(f:X\to
M)=f_*(td_y(X))=f_*(td(X)ch(\Omega^y_X))\,.$$
 If in addition $M$ is
smooth then $$td_y(f:X\to M) =td(M)ch(Rf_*\Omega^y_X)\,.$$ Here we
freely use the Poincar\'e duality isomorphism $H^{BM}_*(M)\simeq
H^{2\dim M-*}(M)$. To compute the Hirzebruch class $td_y(f:X\to
M)$ for a singular $X$ and possibly non proper map one has to
replace $X$ by its {\it geometric resolution} $X_\bullet$, for
example obtained from a cubification in the sense of Guillen and
Navarro Aznar \cite{GNA}. The maps between the members of
$X_\bullet$ are not important, we just have to write $[f:X\to M]\in
K(Var/M)$ as an alternating sum of classes $\sum_{i=0}^{\dim X}
(-1)^i[f_i:X_i\to M]$ with $X_i$ smooth and with the map $f_i:X_i\to M$ being
proper. Then
$$td_y(f:X\to
M)=\sum_{i=0}^{\dim X}(-1)^i td_y(f_i:X_i\to M)\,.$$
 It is convenient to denote the class in $K$-theory as
   $$\Omega^y_{X_\bullet}=\sum_{i=0}^{\dim
 X}(-1)^i\Omega^y_{X_i}\in \bigoplus_{i=0}^{\dim X}K(X_i)[y]$$
and
$$td_y(f:X\to
M)=f_{\bullet*}(td(X_\bullet)\,ch(\Omega^y_{X_\bullet}))\,.$$

\section{Hirzebruch class of a SNC
 divisor complement}\label{divisor}
 The case when $X$ is the complement of a simple normal crossing
 divisor $D\subset M$ is of particular interest, and it is worth to
 give an explicit formula in terms of logarithmic forms. A different formula connecting Hirzebruch class with the sheaf of logarithmic forms was given in  \cite[\S2]{MaSc}.

 Let
 $$D=\bigcup_{k=1}^m D_k\subset M$$
 be the decomposition of $D$ into smooth components. For a subset
 $I\subset\{1,2,\dots,m\}$ let $$D_I=\bigcap_{i\in I} D_i\,.$$
 By inclusion-exclusion formula the $K$-theoretic class $mC_*(X\hookrightarrow M)\in K(M)[y]$ is equal to
 $$\iota_{\bullet*}\Omega^y_{X_\bullet}=\sum_{I\subset\{1,2,\dots,m\}}(-1)^{|I|}\iota_{I*}\Omega^y_{D_I}\in K(M)\,.$$
 Here $\iota_I:D_I\to M$ denotes the inclusion.
 We will find another expression for
 $\iota_{\bullet*}\Omega^y_{X_\bullet}$.
 Let
 $$L\Omega^p_I=\iota_{I*}\Omega^p_{D_I}(\log  D_{>I})$$
 be the sheaf of differential forms on $D_I$ with logarithmic poles
 along the divisor $D_{>I}=\bigcup_{J>I}D_J$ introduced in \cite{De}. We write
 $$L\Omega^y_I=\sum_{p=0}^{\dim D_I}L\Omega^p_I \,y^p\,.$$

\begin{theorem}\label{t1} Let $y=-1-\delta$. With the notation introduced for a simple normal crossing divisor complement
$X=M\setminus D$ we have
\begin{equation}\label{mojmsc}\iota_{\bullet*}\Omega^y_{X_\bullet}=\sum_{I\subset\{1,2,\dots,m\}} \delta^{|I|}\,
L\Omega_I^y\in K(M)[\delta]\,.\end{equation}  The Hirzebruch class of $X$ is equal to
\begin{equation}\label{mojmsch}td_y(X\hookrightarrow M)=\sum_{I\subset\{1,2,\dots,m\}}\delta^{|I|}\,td(M)\,
ch(L\Omega_I^y)\,.\end{equation}
\end{theorem}

\begin{remark}\rm Another way of expressing $\Omega^y_{X_\bullet}$ in terms of sheaves of logarithmic forms was given in \cite[Prop.2.2]{MaSc}:
\begin{equation}\label{msc}\Omega^y_{X_\bullet}=\O_M(-D)\otimes \Omega^y_M(\log D)\,\end{equation}
and
\begin{equation}\label{msch} td_y(X\hookrightarrow M)=td(M)\,ch\big(\O_M(-D)\otimes \Omega^y_M(\log D)\big)\,.\end{equation}
The formulas (\ref{msc}) and (\ref{mojmsc}) are equivalent as we explain below. Moreover the formula (\ref{msch}) can be generalized is the setup of mixed Hodge modules,  see \cite[Proposition 5.2.1]{MSS}.
Let's compare two expressions (\ref{msc}) and (\ref{mojmsc}) for $\Omega^y_{X_\bullet}$, say under assumption that $D$ consists of one component. 
We will skip $\iota$ in the notation:
$$\Omega^y_{X_\bullet}=\Omega^y_M(\log D)+\delta \Omega^y_D=\O_M(-D)\otimes \Omega^y_M(\log D)\,.$$
Since in $K$-theory $\O_M(-D)=\O_M-\O_D$ we have
$$\Omega^y_M(\log D)+\delta \Omega^y_D=(\O_M-\O_D)\otimes\Omega^y_M(\log D)\,.$$
It follows
$$\delta \Omega^y_D=-\O_D\otimes\Omega^y_M(\log D)\,.$$
In general in $K$-theory there is an equality
$$\O_M(-D)=\sum_{I\subset\{1,2,\dots,m\}}(-1)^{|I|}\O_{D_I}\,.$$
and one can show (for example inductively from the equation of the formulas (\ref{msc}) and (\ref{mojmsc})) that
$$(-\delta)^{|I|}L\Omega^y_I=\O_{D_I}\otimes \Omega^y_M(\log D)\,.$$
\end{remark}

We recall that the weight filtration and the whole mixed Hodge
structure for an open smooth variety  were constructed via the
logarithmic complex by Deligne \cite{De}. Theorem \ref{t1} or \cite[Prop 2.2]{MaSc} clarifies the relation of
the Hirzebruch class with the mixed Hodge structure. The general
point of view was presented in \cite[\S4]{BSY}, but the considered case
is fairly explicit.  More general approach reletting mixed Hodge modules and theory of characteristic classes is described in the survey \cite{Scu}.
\s
\noindent{\it Proof of Theorem \ref{t1}.} We will skip $\iota$ in the notation.  The sheaf of logarithmic differential forms
is equipped with the weight filtration.
 The associated quotients of the weight
 filtration in $L\Omega^p_\emptyset=\Omega^p_M(\log D)$ are equal to the
 sheaves of forms on the intersections of divisor components (with shifted gradation),
 therefore we have
 \begin{equation}L\Omega^y_\emptyset=\sum_{I\subset\{1,2,\dots,m\}}y^{|I|}\,\Omega^y_{D_I}\, \,.\label{log2no}\end{equation}
 Similarly
 \begin{equation}L\Omega^y_J=\sum_{I\supset J}y^{|I|-|J|}\,\Omega^y_{D_I}\, \,.\label{log2no2}\end{equation}
From the equations (\ref{log2no}) and (\ref{log2no2}) we find that
  \begin{equation}\Omega^y_J=\sum_{I\supset J}(- y)^{|I|-|J|}\,L\Omega^y_{D_I}\,\,.\end{equation}
By the inclusion-exclusion formula we have
 $$\Omega^y_{X_\bullet}=\sum_{J\subset\{1,2,\dots,m\}}(-1)^{|J|}\Omega^y_{D_J}=
 \sum_{J\subset\{1,2,\dots,m\}}(-1)^{|J|}\sum_{I\supset J}(- y)^{|I|-|J|}\,L\Omega^y_{D_I}\,=$$
$$ =\sum_{I\subset\{1,2,\dots,m\}}\left(\sum_{J\subset I}(-1)^{|J|}\,(- y)^{|I|-|J|}\right) L\Omega^y_{D_I}
=$$ $$=\sum_{I\subset\{1,2,\dots,m\}}\left(
\sum_{k=0}^{|I|}\binom{|I|}k(-1)^{k}\,(-y)^{|I|-k} \right)
L\Omega^y_{D_I} = \sum_{I\subset\{1,2,\dots,m\}}(-1-y)^{|I|}
L\Omega^y_{D_I}\,.$$
\hfill\qed

\begin{example}\label{log1dim}\rm If $D\subset M$ is a smooth divisor
then we have the residue exact sequences
$$0\to\Omega^*_M\hookrightarrow \Omega^*_M(\log D)\stackrel{res}\to \Omega^{*-1}_D\to 0\,.$$
Therefore in $K$-theory we have
\begin{equation}L\Omega^y_{\emptyset}=\Omega^y_M(\log D)=\Omega^y_M+y\Omega^y_D\,.\label{klog}\end{equation}
The decomposition of
the Theorem \ref{t1} takes form of the sum of two components, logarithmic and residual part:
$$\Omega^y_{X_\bullet}=
L\Omega_\emptyset^y+ \delta L\Omega_{\{1\}}^y\,,$$
Indeed
$$L\Omega_\emptyset^y\oplus \delta L\Omega_{\{1\}}^y=\Omega^y_M+y\Omega^y_D-(1+y)\Omega^y_D=\Omega^y_M-\Omega^y_D\,.$$
\end{example}
Let us now derive from formula (\ref{msch}) the Aluffi expression $
c^{SM}(1\!\!1_X)$ via logarithmic tangent bundle.
\begin{corollary} The  formula (\ref{msch}) specializes to the Aluffi formula
\cite{Alog}  for Chern-Schwartz-MacPherson class
$$c^{SMC}(1\!\!1_X)=c(TM(-\log D))
$$
by taking $y=-1$, that is $\delta=0$. \end{corollary}

\begin{proof}  Suppose $x_i$ for $i=1,2,\dots,\dim(M)$ are the Chern roots of $TM$ and $\xi_i$ for $i=1,2,\dots,\dim(M)$ are the Chern roots of $TM(-\log D)$. By (\ref{msch})
$$\widetilde{td}_y(X\hookrightarrow M) =\prod_{i=1}^{\dim(M)}\frac{x_i}{1-e^{-\delta x_i}}\cdot e^{-\delta [D]}\cdot \prod_{i=1}^{\dim(M)}(1-(1+d)e^{\delta \xi_i})\,.$$
Let us compute what is the limit
of this characteristic class
with $\delta\to 0$. We   check that
$$\lim_{\delta\to 0}e^{-\delta [D]}=1\,,\qquad\lim_{\delta\to 0}\frac x{1-e^{\delta x}}(1-(1+\delta)\,e^{\delta \xi})=1+\xi\,.$$
Therefore
$$\lim_{y\to -1}\widetilde{td}_{y}(M)=\prod_{i=1}^{\dim(M)}(1+\xi_i)=c(TM(-\log D))\,.$$
\end{proof}

\begin{remark}\rm
We can easily derive the Aluffi formula from (\ref{mojmsch}) and see which summands contribute to the Chern-Schwartz-MacPherson class. Indeed
\begin{equation}\widetilde{td}_y(X\hookrightarrow M)=\sum_{I\subset\{1,2,\dots,m\}}\delta^{|I|}\,\widetilde{td}(M)\,
{ch}(\Omega^y_{D_I}(\log
D_{>I}))\,,\label{alu}\end{equation} and
${ch}(\Omega^y_M(\log D))$ is the product of the
factors $1+y\,e^{-(y+1) \xi_i}=1-(1+\delta)e^{\delta \xi_i}$.
It follows that
$$\lim_{\delta\to 0}\widetilde{td}(M)\,{ch}(\Lambda_yT^*M(-\log D))=c(TM(-\log D))
$$ The remaining
summands in the formula (\ref{alu}) vanish in the limit since they converge to
$$\lim_{\delta\to 0}\delta^{|I|}\widetilde{td}(D_I)\,{ch}(\Omega^y{D_I}(\log
D_{>I}))= 0^{|I|}c(TD_I(-\log D_{>I}))=0$$ for $|I|>0$.
\end{remark}
\s
Let us examine the specialization $y=0$.
\begin{corollary}  With the introduced notation for a simple normal crossing divisor complement
$X=M\setminus D$ we have
$$td_0(X\hookrightarrow M)=td(M)\,ch(\O_M(-D))=td(M)e^{-D}\,.$$
\end{corollary}

\begin{proof} If $y=0$, then $\delta=-1$ and
$$\Omega^y_{X_\bullet}=\sum_{I\subset\{1,2,\dots,m\}} \delta^{|I|}\,
L\Omega_I^y=\sum_{I\subset\{1,2,\dots,m\}}(-1)^{|I|}\O_{D_I}=\O_M-\O_D$$
and from the exact sequence  $$0\to \O_M(-D)\to \O_M\to \O_D\to
0$$ we obtain the result.\end{proof}

Note that we have
$$td_0(D\hookrightarrow M)=td(M)(1-ch(\O_M(-D)))=td(M)ch(\O_D)\,,$$ which agrees
with the image of the Baum-Fulton-MacPherson class of $D$ in
$H_*(M)$, see \cite[18.3.5]{Fu}. This is not always the case. Only
the varieties with mild singularities have this property, see
\S\ref{bfm}. We will show various explicit examples after having
introduced the local Hirzebruch class for varieties with torus action.

\section{Localization of equivariant homology}
Our goal is to study singularities locally, but the characteristic
classes are global objects. Assume that an algebraic group is acting on an
algebraic variety  and a singular point is fixed. Then when we
localize the characteristic class at this point, we obtain some
nontrivial information. There is a technical inconvenience: the
Hirzebruch class is defined in the homology of the target space
$M$. If $M$ is smooth, then homology can be replaced by
cohomology, and its equivariant version is well developed and
widely known. If $M$ is singular, the Hirzebruch class naturally lives in
equivariant homology. This theory is less developed  but it is present in
the literature (\cite{BL,GKM,BZ}). Equally well one can work with
equivariant Chow groups \cite{EdGr}. Another definition of equivariant homology can be found in \cite[\S3.3]{AFP}.  We will briefly recall the
general theory not assuming that $M$ is smooth. We will
concentrate on the case when a torus is acting.

The definition of equivariant homology by approximation given in \cite{BZ} or in \cite{EdGr} is the
most convenient for us. We fix a decomposition of the torus
$\T=(\C^*)^r$. The universal $\T$-bundle $E\T\to B\T$ is
approximated by the sequence of finite dimensional algebraic
varieties
$$E\T_m=(\C^{m+1}\setminus\{0\})^r\to B\T_m=(\P^m)^r\,.$$
The equivariant homology of a $\T$-variety $M$ is defined by
$$H_{\T,k}(M)=\mathop{\lim_{\longleftarrow}}_{m}H^{BM}_{k+2rm}(E\T_m\times_\T M)\,.$$
The limit is taken with respect to the Gysin maps
$$H^{BM}_{k+2rm}(E\T_m\times_\T M)\to H^{BM}_{k+2r(m-1)}(E\T_{m-1}\times_\T M)\,,$$
and for a fixed gradation it stabilizes.  The equivariant homology
may be nonzero in gradations $\leq 2\,\dim M$, also in negative
gradations. The equivariant homology is isomorphic to the equivariant
cohomology with coefficient in the dualizing sheaf $\D_M$, which
is an equivariant sheaf in the sense of \cite{BL}. Equivariant
homology is a module over the equivariant cohomology of the point
$$H^*_\T(pt)=H^*(B\T)=Sym^*(\Hom(\T,\C^*)\otimes \Q)$$
via the pullback $H^*(B\T)\to H^*(B\T_m)\to H^*(E\T_m\times_\T M)$
composed with the cap product. If $M$ is smooth, then
$\D_M=\C_M[2\,\dim M]$ and
$$\cap[M]:H^{2\,\dim M-k}_\T(M)\to H_{\T,k}(M)$$
is an isomorphism.  The localization theorem as stated in
\cite[Theorem 6.2]{GKM} says that up to a $H^*_\T(pt)$-torsion all the
information about $H_{\T,k}(M)$ is encoded at the fixed points:
the restriction map
$$H_{\T,*}(M)= H^{-*}_\T(M;\D_M)\to H^{-*}_\T(M^\T ;(\D_M)_{|M^\T })\,,$$
has torsion kernel and cokernel. The image is of special interest
when the fixed points are isolated. Nevertheless for singular
variety the groups
$H^{-*}_\T(\{p\};(\D_M)_{|p})=H_{\T,*}(M,M\setminus \{p\})$ might
be complicated.

\begin{example}\rm Suppose that $\T=\C^*$ and $p\in M$ is an isolated
fixed point. Let $U$ be a conical  neighbourhood of $p$, which is
invariant with respect to $S^1\subset \C^*$. Then
$$H^{-*}_\T(\{p\};(\D_M)_{|p})=H_{\T,*}(\overline{U},\partial
U)\,.
$$ The exact sequence of
the pair $(\overline{U},\partial U)$ gives us some information
about the cohomology of the point with coefficients in $\D_M$. In this exact sequence we have
$H_{\T,*}(\overline
{U})=H^{-*}_\T(pt)$ and $H_{\T,*}(\partial U)=H_*(\partial U/S^1)$,
since the action of $S^1$ on $\partial U$ has finite isotropy
groups. For
$k\leq \dim M$ we have {$$
 0\to H_{\T,2k+1}(\overline{U},\partial U)\to
  H_{2k}(\partial U/S^1)\to
  \Q\to
 H_{\T,2k}(\overline{U},\partial U)\to
 H_{2k-1}(\partial U/S^1)\to
 0\,.
 $$}

\end{example}

In general if the fixed points are isolated, then the difference
between $H^*_\T(\{p\})$ and $H^{*}_\T(\{p\};(\D_M)_{|p})$   is
measured by  $H_{\T,*}(\partial U;(\D_M)_{|\partial U})$. This is a torsion
$H^*_\T(pt)$-module since there are no fixed points of $\T$ on
$\partial U$. Finally,
the localization theorem
\cite[Theorem 6.2]{GKM} for isolated fixed points and the dualizing sheaf $\D_M$
takes form:
\begin{theorem}[Localization in equivariant homology] For any $\T$-variety with isolated fixed points
$${\mathcal S}^{-1}H_{\T,*}(M)\simeq \bigoplus_{p\in M^T} {\mathcal S}^{-1}H^{-*}_\T(p)\,,$$
where ${\mathcal S}$ is the multiplicative system generated by nonzero
characters $$w\in H^2_\T(pt)=Hom(\T,\C^*)\otimes \Q\,.$$
\end{theorem}

The situation simplifies when $M$ is smooth. One can explicitly write down  the inverse map of the restriction $H_\T^*(M)\to H_\T^*(M^\T)$. Then
$H_{\T,*}(M,M\setminus \{p\})$ is a free $H^*_\T(pt)$ module with
one generator in the gradation $2\,\dim M$. Let us identify
$H_{\T,*}(M,M\setminus \{p\})$ with $H^*_\T(pt)$ by Poincar\'e
duality. In addition, when $M$ is smooth and complete then
$H_{\T,*}(M)=H^{2\dim M-*}_\T(M)$ is a free module over
$H^*_\T(pt)$ (\cite[Theorem 14.1]{GKM}), hence the restriction to
the fixed points is injective. Atiyah-Bott formula \cite[Formula 3.8]{AB}, \cite{BV} (or
\cite{EdGr} for Chow groups) is a recipe how to recover the class
from its restriction to the fixed points:
 \begin{theorem}[Localization Formula]\label{locfo} Let $M$ be a smooth and complete algebraic
 variety. Assume that $M^\T$ is discrete and let $i_p:\{p\}\to M$
 be the inclusion. Then for any equivariant class $a\in H^*_\T(M)$
 $$a=\sum_{p\in M^\T}i_{p*}\frac{i^*_p(a)}{eu(p)}\,,$$
where $eu(p)\in H^{2\dim M}_T(\{p\})$ is the product of  weights
appearing in the tangent representation $T_pM$.\end{theorem}
The resulting {\it Integration Formula}  for isolated fixed points takes form
\begin{equation}\int_Ma=\sum_{p\in M^\T}\frac {i^*_p(a)}{eu(p)}\label{int}\end{equation}

 \begin{remark}\rm Some
generalizations for singular spaces are available, but additional
assumption about the fixed point set are needed \cite[Proposition
6]{EdGr}.\end{remark}

\section{Equivariant version of Hirzebruch class} \label{eqho}

Similarly to \cite{BZ,Oh} we propose  a definition of the equivariant Hirzebruch class of an
 equivariant map $X\to M$. It is based on the observation that the varieties $B\T_m=(\P^{m})^r$ which approximates $B\T$ and $E\T_m\times_\T X$ approximating Borel construction
have their own  Hirzebruch classes.

\begin{definition}[Of the equivariant Hirzebruch class]
$$td_y^\T (X\to M)=\mathop{\lim_{\longleftarrow}}_{m}\left(td_y(E\T_m\times_\T X\to E\T_m\times_\T M)\cap
(\pi^*td_y(B\T_m))^{-1}\right)\,,$$
where  $\pi:E\T_m\times_\T M \to B\T_m$ is the canonical
 projection. Note that since we pass to the limit with finite dimensional skeleta of $B\T$,  we obtain a class which might be nontrivial in infinitely many gradations:
 $$td_y^\T (X\to M)\in\prod_{k=-\infty}^{2\dim M}\Big(H_{\T,k}(M)\otimes \Q[y]\Big)\,.$$
 We omitting completion in the notation and write $H_{\T,*}(X)$ for this group.
\end{definition}
The limit in the definition stabilizes and for a bounded range of gradations. It is
enough to perform all computations for a sufficiently large $m$. If $M$ is smooth then  $$H_{\T,2\dim M-*}(X)\simeq H^*_\T(M)=\prod_{k=0}^{\infty}H_{\T}^k(M)$$ by Poincar\'e duality.
If $X=M$ then the Hirzebruch class $td_y(X\to M)\in  H^*_\T(M)[y]$ can be computed as the characteristic class of the equivariant tangent bundle. Note that only finitely many powers of $y$ appear in the expression for that characteristic class. For arbitrary singular $X$ the class $td_y^\T (X\to M)$ is a combination of finitely many classes $td_y^\T (X_i\to M)$ for smooth $X_i$, therefore also in that case only finitely many powers of $y$ appear. This means that
 $$td_y^\T (X\to M)\in \left(\prod_{k=-\infty}^{2\dim M}H_{\T,k}(M)\right)\otimes \Q[y]\,.$$
If $M=pt$ we do not obtain any additional information. The action of $\T$ is not reflected by $\chi_y$-genus:

\begin{theorem}[Rigidity for singular varieties] \label{rigidity} The equivariant Hirzebruch class $td_y^\T (X\to
pt)\in H_{\T,*}(pt)[y]$  is trivial in nonzero gradations and
$$td_y^\T (X\to pt)=\chi_y(X)\in H_{\T,0}(pt)[y]=\Q[y]\,.$$
\end{theorem}
\begin{remark}\rm This property for smooth varieties
 was studied in \cite{Mu, To}. The series $td_y(x)\in \Q[[x]]$ gives rise to an universal rigid characteristic class. To prove that $td_y(x)$ is universal suppose  a formal series $f(x)=\sum_{n=0}^\infty a_nx^n$ defines a rigid characteristic class, i.e. the equivariant $f$-genus of a smooth variety $X$ vanishes in positive gradations. Rigidity implies some relations between coefficients of $f$. It is easy to check, that setting $X=\P^1$ or $\P^2$ with the standard torus actions the conditions imposed on  $f$ leave only three degrees of freedom: $f$ is determined by the coefficients $a_0$, $a_1$ and $a_2$. It follows that  up to a constant or re-scaling $x$ the series $f(x)$ has to be equal $td_{y_0}(x)$ for some $y_0\in \Q$.\end{remark}
We  derive rigidity for singular algebraic varieties from the motivic properties of the Hirzebruch class.

\s\noindent{\it Proof of Theorem \ref{rigidity}.} The canonical projection $E\T_m\times_\T X\to E\T_m\times_\T pt=B\T_m$
is a locally trivial fibration in Zariski topology therefore in $K(Var/B\T_m)$ the classes
$[E\T_m\times_\T X\to B\T_m]$ and $[B\T_m\times X\to B\T_m]$ are
equal. By the product property of the Hirzebruch class we have
$$td_y(E\T_m\times_\T X\to B\T_m)=\chi_y(X)td_y(B\T_m)\,.$$
The conclusion follows.\hfill\qed

We note that if  $X$ admits a decomposition into affine spaces as
often happens for $\T$-varieties, then $td_y(X\to pt)$ is easy.
Namely:
$$td^\T_y(X\to pt)=\chi_y(X)=\sum_{i=0}^{\dim X} a_i(-y)^i\,,$$
where $a_i$ is the number of $i$ dimensional cells.
Similarly when
$X$ is a toric variety, then
$$td_y^\T(X\to pt)=\chi_y(X)=\sum_{i=0}^{\dim X} b_i(-(y+1))^i\,,$$
where $b_i$ is the number of $i$ dimensional orbits. On the other
hand the local contribution coming from a singular point to the
global class might be fairly complicated.

Suppose, that $M$ is smooth and complete. Assume, that  $\T$ has
only a finite number of fixed points on $M$. Then according to
Atiyah-Bott or Berline-Vergne formula (\ref{int})
\begin{equation}
\chi_y(X) =\sum_{p\in M^\T}\frac{td^\T_y(f:X\to M)_{|p}}{eu(p)}=
\sum_{p\in M^\T}\frac{ch^\T(mC_*(f:X\to M))_{|p}}{ch^\T(\O_{\{p\}})}
\,.\label{hirint}\end{equation}  As it was explained in the introduction, instead of computations in
cohomology we can apply localization theorem for equivariant K-theory,
\cite{Se}. This alternative point of view does not influence the computations, which are done
in the ring of Laurent series.

 Let us give an
example which shows how the local Hirzebruch class at a singular point can be computed by global Localization Theorem \ref{loc}, provided that at the remaining fixed points the variety are smooth. In certain cases this method can be applied in much more general situations.
It was applied in \cite{We} to compute Chern-MacPherson-Schwartz classes of determinantal variety.

\begin{example}\rm \label{stozLGe} Consider the torus $\T =(\C^*)^2 $ acting on $\P^3$ by the formula
$$( \xi_1 ,\xi_2 )\cdot [x_0\, : \,x_1\, :\, x_2\, :\, x_3]=[x_0\, :\,\xi^2_1 x_1\, :\, \xi^2_2 x_2\, :\, \xi_1\xi_2x_3 ]\,.$$
The action preserves the projective cone over the quadric
$$X = \{[x_0\, :\, x_1\, :\, x_2\, :\, x_3]\subset \P^3\,
|\, x_1x_2 - x_3^2 =0\}\, .$$ Since $X$ has a decomposition into
affine cells it is immediate to compute the $\chi_y$-genus
$$\chi_y(X)=1 - y + y^2\,.$$
The variety $X$ has three fixed points: $[0:1:0:0]$ and
$[0:0:1:0]$ are smooth and $[1:0:0:0]$ is singular (the $A_1$
singularity). Let us denote the characters $\T \to \C^*$
$$t_1 :(\xi_1,\xi_2) \mapsto \xi_1,\qquad t_2 :(\xi_1,\xi_2) \mapsto \xi_2\,.$$
In the standard affine neighbourhood of the point $[0:1:0:0]$ there are the
coordinates $\frac{x_0}{ x_1}$, $\frac{x_2} {x_1}$,
$\frac{x_3}{x_1}$. The tangent representation of $\T$ at $[0:1:0:0]$ has the characters
$-2t_1\, ,\, 2(t_2 - t_1)\, ,\,t_2 - t_1$. The weight at the point $[0:0:1:0]$ differ by a switch
of subscripts.
The equation of $X$ is $\frac{x_2} {x_1}+
\left(\frac{x_3} {x_1}\right)^2=0$. The linear part has weight
$2(t_2 -t_1)$, therefore $[X]_{|[0:1:0:0]} =2(t_2 -t_1)$.

\vskip 20pt

\hskip19pt $[0:1:0:0]$

\hskip19pt tangent weights:

\hskip19pt $-2t_1\,,~t_2-t_1$

\hskip19pt normal weight:

\hskip19pt $2(t_2-t_1)$ \vskip -30pt

 \hskip270pt $[0:0:1:0]$

\hskip270pt tangent weights:

\hskip270pt $-2t_2\,,~t_1-t_2$

\hskip270pt normal weight:

\hskip270pt $2(t_1-t_2)$

\begin{center}\vskip-100pt\pgfdeclareimage[height=7cm]{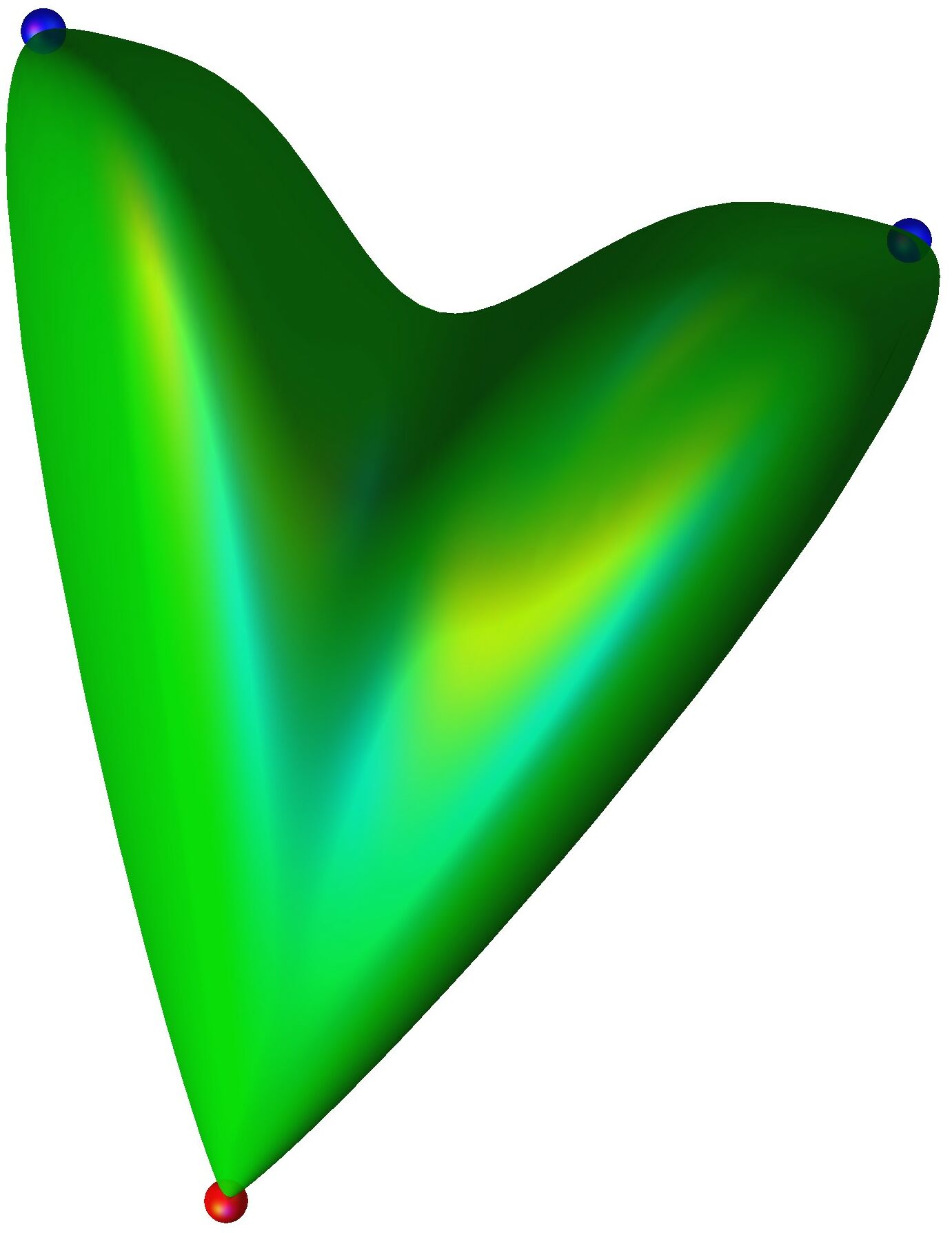}{stkwa}
\pgfuseimage{stkwa}\end{center}  \hfil${[1:0:0:0]}$ -- singular point
\s

\noindent The
Hirzebruch class at the point $[0:1:0:0]$ is equal to
$$td^\T_y(X)_{|[0:1:0:0]}=\frac{-2t_1(1+ye^{2t_1})}{1-e^{2t_1}}\cdot \frac{(t_2-t_1)(1+ye^{t_1-t_2})}{1-e^{t_1-t_2}}\,.$$
The image under the inclusion into $M = \P^3$ is equal
$$td^\T_y(X\to M)_{|[0:1:0:0]}=2(t_2-t_1)\frac{-2t_1(1+ye^{2t_1})}{1-e^{2t_1}}\cdot \frac{(t_2-t_1)(1+ye^{t_1-t_2})}{1-e^{t_1-t_2}}\,.$$
Similarly we compute $td^\T_ y (X)_{|[0:0:1:0]}$.
The variables
$t_i$ cancel out in the formula (\ref{hirint}), it remains only
$T_i = e^{-t_i}$. After this  substitution we have
{\begin{align*}&y^2-y+1=\\ &=\frac{\left(1+y\frac{1}{{T_1
   }^2}\right)
   \left(1+y\frac{{T_2}
   }{{T_1}}\right)}{\left(1-\frac{1}{{T_1}^2}\right)
   \left(1-\frac{{T_2}}{{T_1}}\right)}-\frac{\left(1+y\frac{1}{{T_2}^2}\right) \left(1+y\frac{{T_1}
   }{{T_2}}\right)}{\left(1-\frac{1}{{T_2}^2}\right)
   \left(1-\frac{{T_1}}{{T_2}}\right)}
   +\frac{ch^\T(mC_*(f:X\to M))_{|[0:0:0:1]}}
   {(1-T_1^2)(1-T_2^2)(1-T_1T_2)}\,.\end{align*}}
Simplifying the expression we find the formula for
$ch^\T(mC_*(f:X\to M))_{|[0:0:0:1]}$:
\begin{equation}\label{stozLG}(1-{T_1}^2 {T_2}^2)
  +y
   ({T_1}+{T_2})^2(1-{T_1} {T_2})+
  y^2{T_1} {T_2}
   (1-{T_1}^2 {T_2}^2)\end{equation}
In particular here for $y=0$ $$ch^\T(mC_0(f:X\to M))=ch^\T(1-\O(-X))=ch^\T(\O_X).$$
\end{example}

\section{Local Hirzebruch class for SNC divisor}

We want to describe $td_y^\T(X\hookrightarrow M)$ of a singularity
germ of a subvariety $X\subset M$. Let us concentrate on the case
when $M$ is smooth and $X$ is an open subset. To compute this
local invariant directly we resolve singularities, that is we find a
proper map $f:\widetilde M\to M$ such that $X\simeq \widetilde
X=f^{-1}(X)$ and $\widetilde M\setminus \widetilde X$ is a simple
normal crossing divisor. Then $td_y^\T(X\hookrightarrow
M)=f_*(td_y^\T(\widetilde X\hookrightarrow\widetilde M))$.
Therefore it is crucial to understand the situation when
$D=M\setminus X$ already is SNC divisor.

The local Hirzebruch class is invariant with respect to  analytic
changes of coordinates, so we can assume that we have coordinates
preserved by the torus action. The divisor $D$ which is a union of coordinate hyperplanes defined by the equation
$\prod_{i=1}^kx_i=0$. The  weights of  $\T$ acting on coordinates
will be denoted by $w_1,w_2,\dots, w_n$. It is immediate to  write
down the Hirzebruch classes of the basic constructible sets:

\begin{proposition} Let $D=\{0\}\subset \C=M$, $X=\C\setminus D$, and let $\T=\C^*$ acts on $\C$ with the weight $w$. Set $\delta=-1-y$, $S=e^{-iw}$. Then
\begin{equation}td^\T_y(D\hookrightarrow \C)_{|0}=w\label{punkt}\end{equation}
\begin{equation}\label{L}
td^\T_y(\C)_{|0}
=w\frac{1+y\,e^{-w}}{1-e^{-w}}=w\frac{\delta+S+\delta S}{S}\end{equation}
\begin{equation}td^\T_y(X\hookrightarrow \C)_{|0}
=w\frac{(1+y)\,e^{-w}}{1-e^{-w}}=w\frac{\delta(1+ S)}{S}\label{snc1}\end{equation}
\begin{equation}td^\T(\C)ch^\T(L\Omega^y_\emptyset)_{|0}
=w\frac{1+y}{1-e^{-w}}=w\frac{\delta}{S}\label{snclog}\end{equation}
\end{proposition}

\begin{proof} The formula (\ref{punkt}) holds since $w\in H^2(B\T)=Hom(\T,\C)$ is the Euler class of the normal bundle of $D$.
 The formula (\ref{L}) is by the straight forward substitution.
To prove (\ref{snc1}) we use additivity.
$$td^\T_y(X\hookrightarrow \C)_{|0}=td^\T_y(\C)_{|0}-td^\T_y(D\hookrightarrow \C)_{|0}=w\left(\frac{\delta+S+\delta S}{S}-1\right)=
w\frac{\delta(1+ S)}{S}\,.$$
 The formula (\ref{snclog}) follows from (\ref{klog})
 \begin{align*}td^\T(\C)ch^\T(L\Omega^y_\emptyset)_{|0}&=td^\T_y(\C)_{|0}+y\,td^\T_y(D\hookrightarrow \C)_{|0}
 \\&
=w\left(\frac{\delta+S+\delta S}{S}-(\delta+1)\right)=w\frac{\delta}{S}\end{align*}
\end{proof}

By the product property of the Hirzebruch class we obtain
\begin{corollary}\label{divisor2} Let $D=\{(x_1,x_2,\dots, x_n)\in\C^n\,|\,\prod_{i=1}^kx_i=0\}$ be a simple normal crossing divisor and $X=M\setminus D$.
Set $S_i=e^{-w_i}-1$ and $\delta=-1-y$, where $w_i$ is the weight
of $\T$ acting on $i$-th coordinate. For arbitrary
$n$ and $k$ we have
\begin{align}
td^\T_y(\C^n)_{|0}
&=eu(0)\prod_{i=1}^n\frac{\delta+S_i+\delta S_i}{S_i}\\
td^\T_y(X\hookrightarrow \C^n)_{|0}
&=eu(0)\delta^k\prod_{i=1}^k\frac{1+ S_i}{S_i}\prod_{j=k+1}^n\frac{\delta+S_j+\delta S_j}{S_j}\\
td^\T(\C^n)ch^\T(L\Omega^y_\emptyset)_{|0}
&=eu(0)\delta^k\prod_{i=1}^k\frac{1}{S_i}\prod_{j=k+1}^n\frac{\delta+S_j+\delta S_j}{S_j}\end{align}
\end{corollary}
\noindent Multiplying by $\frac1{eu(0)}\prod_{i=1}^nS_i=(-1)^ntd^\T(\C^n)^{-1}$ we obtain
an expression for $$(-1)^n ch^\T mC_*(X\hookrightarrow M)_{|0}$$ which
is a polynomial with nonnegative coefficients in  $\delta$ and
$S_i$.
We will examine various examples and we will see that the
positivity in the $S_i$ and $\delta$ variables is preserved for a
large class of singularities.

\section{Whitney umbrella: an example of computation via resolution}

\begin{example}\rm  Consider the torus $\T=(\C^*)^2$ acting on $\C^3$ with weights
$$w_1=t_1+t_2\,,\,w_2=t_1\,,\,w_3=2t_2\,.$$
The weights are non\-ne\-ga\-tive combinations of the
weights $t_i$.
It follows that for any invariant subvariety its fundamental class
in equivariant cohomology is a non\-ne\-ga\-tive combination of the
monomials in $t_i$ (Theorem \ref{posy}). We will observe a similar effect for the
Hirzebruch class of the Whitney umbrella
 $$X=\{(x_1,x_2,x_3)\in \C^3 \,|\,x_1^2-x_2^2x_3=0\}\,.$$
 Let $$Z=\{(x,y,z)\in \C^3 \,|\, x=0,y=0\}\quad \text{and}\quad X^o=X\setminus
 Z\,.$$
and let
$$f:\widetilde{X}=\C^2\to \C^3\,,\quad f(u,v)=(uv,u,v^2)\,,$$
be the resolution of the Whitney umbrella,  $\widetilde
X^o=f^{-1}X^o=\{u\not=0\}$. We have
\begin{align*}td_y^\T(X\to\C^3)&
=td_y^\T(X^o\to\C^3)+td_y^\T(Z\to\C^3)\\
&=f_*td_y^\T(X^o\to\C^2)+td_y^\T(Z\to\C^3)\\
&=2(t_1+t_2)\frac{t_1(1+y)\,e^{-t_1}}{1-e^{-t_1}}
\frac{t_2(1+y\,e^{-t_2})}{1-e^{-t_2}}+
(t_1+t_2)t_1\frac{2t_2(1+y\,e^{-2t_2})}{1-e^{-2t_2}}
\,.\end{align*} In the variables $T_i=e^{-t_i}$ {$$
\frac{td_y^\T(X\to\C^3)}{2t_1t_2(t_1+t_2)}=\frac{1 + T_1 T_2 +
y(T_1 + 2 T_1 T_2 + T_2^2)  +y^2 (T_1 T_2  +
   T_1 T_2^2) }{(1-T_1)(1-T_2^2)}$$}
   and in the variables $S_i$
{  \begin{align*}&\frac{S_1 S_2 (2 + S_2)  +
 \delta\cdot (S_1 + 2 S_2 + 4 S_1 S_2 + S_2^2 + 2 S_1 S_2^2)+ \delta^2 (1 + S_1) (1 + S_2) (2 + S_2)}{S_1(S_2^2+S_2)}\,.\end{align*}}
For the complement of the Whitney umbrella we obtain
$$\frac{\delta (1 + S_1) (1 + S_2)\left(S_1 S_2 (2 + S_2)+\delta S_2 \,(1 + 3 S_1 + S_2 + 2 S_1 S_2)+\delta^2 (1 + S_1) (1 + S_2)^2\right)}{(S_1S_2+S_1+S_2)S_1(S_2^2+S_2)}$$
The expressions have nonnegative coefficients and multiplying by the factor
 $${(S_1S_2+S_1+S_2)S_1(S_2^2+S_2)}=(-1)^3td^\T(\C^3)^{-1}$$ we obtain the formula for
$mC_*(X\hookrightarrow\C^3)$ with globally predicted signs $(-1)^3$.

There is another issue which is the same as in the normal crossing
case: for $y=0$ the Todd class of the Whitney umbrella is equal to
$$
td^\T( \C^3)\cdot(1-e^{2(t_1+t_2)})=td^\T( \C^3)\cdot
ch^\T(\O_X)\,.$$
\end{example}

\section{Conical singularities}\label{conical}
First let us recall the basic calculational properties of Todd
class. Stably the tangent bundle $T\P^{n-1}$  is equal to
$\O(1)^{\oplus n}$ therefore $td(\P^{n-1})=\left(\frac
h{1-e^{-h}}\right)^n$, where $h=\O(1)$. We have
$$1=\int_{\P^{n-1}}td(T\P^{n-1})=\text{coefficient of $h^{n-1}$ in
 }\frac{h^n}{(1-e^{-h})^n}=Res_{h=0}\frac1{(1-e^{-h})^n}$$
For convenience let us set  $U=e^{-h}-1$. We have
$$Res_{h=0}\frac 1{U^n}=(-1)^n$$
for $n>0$.

We note the following easy facts
 \begin{lemma}
$$Res_{h=0}\frac{(1+U)^k}{U^n}=-{k-1\choose n-1}\,.$$
\end{lemma}
This computation can be done elementary, but having in mind
Riemann-Roch theorem and Serre duality we see immediately
$$Res_{h=0}\frac{(1+U)^k}{(-U)^n}=\chi(\P^{n-1};\O(-k))=(-1)^{n-1}\chi(\P^{n-1};\O(k-n))$$
which is equal to $(-1)^{n-1}h^0(\P^{n-1};\O(k-n))=\dim
\left(Sym^{k-n}(\C^n)\right)$ for large $k\geq n$, but the formula
for Euler characteristic holds for all $k$.

\begin{lemma}\label{sum} Let $S$ be an independent variable. For $0\leq k < n$ we have
\begin{equation}Res_{h=0}\frac{(1+U)^{k +1}}{U^n(S-U)}=
 -\frac{(1+S)^k}{S^n}\,.\label{eqsum}\end{equation}

\end{lemma}

\begin{proof}
\begin{align*}Res_{h=0}\frac{(1+U)^{k +1}}{U^n(S-U)}&
 =Res_{h=0}\left(\sum_{i=0}^\infty\frac{(1+U)^{k +1}}{S^{i+1}U^{n-i}}\right)\\
 &=-\sum_{i=0}^{n-1} {{k }\choose{n-i-1}} \frac1{S^{i+1}}\\
 &=-\frac{1}{S^n}\sum_{i=0}^{n-1}{{k }\choose{n-i-1}}{S^{n-i-1}}\\
&=-\frac{(1+S)^k }{S^n}
\end{align*}\end{proof}

The equality (\ref{eqsum}) for $k \geq n$ is much more
complicated, since summing  the  terms ${k \choose{n-i-1}}
{S^{n-i+1}}$ up to $n-1$ we do not have the full binomial
expansion of $(1+S)^k $. \s
Let $Y\subset \P^{n-1}$ be a subvariety (in fact $Y$ can be any constructible subset)
We give a formula for the Hirzebruch
class of the affine cone over $Y$ with the vertex removed.
 Let
$\T=\C^*$ acting on $\C^n$ by multiplication of coordinates.

\begin{proposition}\label{prco} Let $Y\subset \P^{n-1}$. Suppose that
$$td_y(Y\hookrightarrow \P^{n-1})=
h^n\frac {f(U)}{U^n}$$ for a polynomial $f\in \Z[y][U]$ which has
degree in $U$ less or equal than $n-1$. The Hirzebruch class of the cone without the vertex
$X=C_Y\setminus \{0\}\subset \C^n$ is given by the formula
$$td_y^\T(X\hookrightarrow \C^n)=t^n\delta\left(
\chi_y(Y)-\frac {f(S)}{S^n}\right)
$$
where $S=e^{-t}-1$.
\end{proposition}

This proposition is an extension of the  formula for the
Chern-Schwartz-MacPherson class given in \cite{We}, which in turn
originates from \cite[Lemma 3.10]{AMfe}.

\begin{proof} For simplicity we assume that $f(U) =(1+U)^k $,
for $k=0,1,\dots, n-1$.  These polynomials form a basis of
$H^*(\P^{n-1})$.
 Let
$\widetilde\C^n$ denote the blowup of $\C^n$ at $0$ and let
$E=\P^{n-1}$ be the exceptional divisor, its first Chern class is
equal to $t-h$. Denote by $\widetilde X$ the proper transform of $X$ in $\widetilde \C^n$.  Hence
$$td_y^\T(\widetilde
X\hookrightarrow \widetilde\C^n)_{|E}=td_y^\T(Y\hookrightarrow
\P^n)\cdot (td_y^\T(N_E\widetilde\C^n)-[E]_{|E})\,,$$
and the characteristic class $td^\T_y$ applied to the normal bundle of the exceptional divisor $td^\T_y(N_E\widetilde\C^n)$ is equal to
$$td^\T_y(N_E\widetilde\C^n)=(t-h)\frac{1+y\,e^{h-t}}{1-e^{h-t}}=(t-h)\frac{1+y(1+S)/(1+U)}{1-(1+S)/(1+U)}\,.$$
As in   Example \ref{log1dim} we  decompose:
$$td_y^\T(\widetilde
X\hookrightarrow \widetilde\C^n)_{|E}=[LOG]+[RES]\,,$$
\begin{align}&[LOG]=td_y^\T(Y\hookrightarrow
\P^n)\cdot \left(td^\T_y(N_E\widetilde\C^n)-(1+\delta)[E]_{|E}) \right) \,,\\
&[RES]=td_y^\T(Y\hookrightarrow
\P^n)\cdot\delta[E]_{|E})\,.\end{align}
Since$$ \frac{1+y  (1+S)/(1+U)}{1-(1+S)/(1+U)}-(1-\delta)=\frac{ \delta (1 + U)}{S - U},$$
the logarithmic part of $td_y^\T(\widetilde X\hookrightarrow \widetilde \C^n)_{|E}$ is equal to
$$[LOG]=\delta(t-h)\frac {1+U}{S-U}td_y^\T(Y\hookrightarrow \P^n)$$
 and the residual part is just push-forward of the original Hirzebruch class multiplied by $\delta$
$$[RES]=\delta(t-h)\,td_y^\T(Y\hookrightarrow \P^n)\,.$$ This decomposition resembles
the decomposition of one dimensional space in Proposition
\ref{divisor2}  (\ref{snclog}).
 We will compute the Hirzebruch class using functoriality (\ref{funkt}) and
the local integration formula
 \begin{align*}\frac 1{t^n} td_y^\T( X\hookrightarrow \C^n)_{|\{0\}}
 &=\int_{E}\frac 1{t-h}\left(td_y^\T(\widetilde
 X\hookrightarrow \widetilde\C^n)_{|E}\right)=\int_{\P^{n-1}}\frac{[LOG]+[RES]}{t-h}
 \end{align*}
The integral of the logarithmic part is computed due to Lemma
\ref{sum}:
$$\int_{\P^{n-1}}\frac{ [LOG]}{t-h}=
\delta\,Res_{h=0}(1+U)^k\frac {1+U}{S-U}=-\frac{(1+S)^k}{S^n}\,.$$ The
second ingredient is
$$\int_{\P^{n-1}}\frac{[RES]}{t-h}=\int_E\delta\,td(\P^{n-1}) (1+U)^k =\delta\,\chi_y(Y)\,.$$
The formula of the Proposition follows.
\end{proof}

Let us concentrate on the case $y=0$. Now let $X$ be the closed cone over a smooth  hypersurface
$Y\subset \P^{n-1}$ of degree $d$. The Todd class of $Y$ is equal
to $$td(\P^{n-1})ch^\T(\O_Y)=\frac{h^n}{(-U)^n}(1-(1+U)^d)\,,$$ hence
$$td((\P^{n-1}\setminus Y)\hookrightarrow \P^{n-1})=\frac {h^n}{(-U)^n}(1+U)^d\,,$$
and from Proposition \ref{prco} for $d<n$
$$td^\T((\C^n\setminus X)\hookrightarrow \C^n)=\frac {t^n}{(-S)^n}(1+S)^d\,.$$
We conclude that
$$td^\T(X\hookrightarrow \C^n)=\frac {t^n}{(-S)^n}(1-(1+S)^d)\,.$$
We have used the assumption of Proposition \ref{prco} that $d<n$,
and by Lemma \ref{sum} the formula holds also for $d=n$.

\begin{corollary} \label{conicalra} Let $d\leq n$, then
$$td^\T(X\hookrightarrow \C^n)=td^\T(\C^n)\cdot ch^\T(\O_X)\,.$$
\end{corollary}
In fact as explained in \cite[Example 3.2]{BSY} and discussed in \S\ref{bfm} the conclusion of Corollary holds
for Du Bois singularities.\s
We note another phenomenon  which is valid
only for the degree $d=2$:

\begin{proposition} \label{kwadr} If $X=Q_n\subset\C^n$ is a quadratic cone then both
$$td^\T((\C^n\setminus Q_n)\hookrightarrow \C^n)\quad
\text{and} \quad  td^\T( Q_n\hookrightarrow \C^n)$$ when expanded in
the variables $S$ and $\delta$ have nonnegative coefficients. The
same holds for the logarithmic part of  $td^\T(\C^n\setminus Q_n\hookrightarrow
\C^n)$.
\end{proposition}

\proof The proof is by induction. We check directly that for $n>2$
(we omit straight-forward computations, see \cite{MiWe})
$$td_y^\T(Q_n\to\C^n)=
\frac{\delta  S (2 + S) (\delta  + S + \delta  S)^{n-2}}{S^n}+(1+\delta)td_y^\T(Q_{n-2}\hookrightarrow\C^{n-2})$$
\begin{align*}td_y^\T((\C^{n}\setminus Q_{n})\to\C^{n})=&
\frac{\delta ^2 (1 + S)^2 (\delta  + S + \delta  S)^{n-2}}{S^n}\\+&
(1+\delta)td_y^\T((\C^{n-2}\setminus Q_{n-2})\to\C^{n-2})\end{align*}
as functions in $S$ and $\delta$.
The positivity of the classes $td_y^\T(Q_n\to\C^n)$ and $td_y^\T((\C^{n}\setminus Q_{n})\to\C^{n})$ follows. The logarithmic part of  $td_y^\T((\C^{n}\setminus Q_{n})\to\C^{n})$ is just the sum of terms with negative exponents of $S$, therefore it has nonnegative coefficients as well.
\qed

\begin{example}\rm We have said that for the quadratic cones the class $td^\T_y(Q_n\hookrightarrow \C^n)$ is positive in $\delta$ and $S$ variables.
Another example of a homogenous hypersurface which has the positive Hirzebruch class is the cone over an elliptic curve
$E$:
$$td^\T_y(Cone(E)\hookrightarrow \C^3)=\frac1{S^2} (3 \delta^2 + 3
\delta^2 S + S^2)\,,$$

 $$td^\T_y((\C^3\setminus Cone(E))\hookrightarrow \C^3)=\frac\delta{S^3} (1 + S) (\delta^2 + 2 \delta^2 S + 3 S^2 + 3 \delta S^2 + \delta^2 S^2))\,.$$
All other hypersurfaces in higher dimension spaces or of higher degrees do not have positive Hirzebruch classes.
\end{example}
\section{Toric varieties}\label{toroz}
Let us come back to the general situation considered in
\S\ref{divisor} of the SNC divisor complement, but with the additional
assumption that the fixed points of the torus action are
contained in the 0-dimensional stratum of the divisor:

\begin{lemma}\label{llemm} Suppose that $M$ is a  smooth variety of dimension $n$,  $X=M\setminus D$ is the complement of a SNC divisor. Assume
$$M^\T=\bigcup_{|I|=n} D_I$$ with the  notation of \S\ref{divisor}.
Then
$$td_y^\T(X\hookrightarrow M)=(y+1)^n td^\T(X\hookrightarrow M)\,,$$
up to $H^*_\T(pt)$-torsion.
\end{lemma}
Note that  we study the Hirzebruch class mainly in two cases:
\begin{itemize}
\item global class in $H^*_\T(M)$ for $M$ smooth and complete, \item the restriction
to a fixed point.\end{itemize} In both situations  there is
no $H^*_\T(pt)$-torsion.

\begin{proof} By Localization Theorem \ref{locH} it is enough to verify the equality at the fixed points. Indeed, if $p\in D_I=\bigcap_{i\in I} D_i$
$$td_y^\T(X\hookrightarrow M)_{|p}=\prod_{i\in I}\frac{w_i(y+1)e^{-w_i}}{1-e^{-w_i}}=(y+1)^n\prod_{i\in I}\frac{w_ie^{-w_i}}{1-e^{-w_i}}$$
by formula (\ref{snc1}), where $w_i=[D_i]_{|p}$.\end{proof}

The above Lemma extends to the case of a general pair $(M,D)$, provided
that it has a resolution with desired properties. In particular we obtain
for toric varieties:

\begin{corollary} \label{toric} Let $X$ be a toric variety, then
$$td_y^T(id_X)=\sum_{\rm{orbits}}(1+y)^{\dim \O_\sigma}td^T(O_\sigma)\,.$$
\end{corollary}

\begin{proof} For the basic properties of toric varieties we refer the reader to  \cite{Futor}. We argue that for each orbit $\O_\sigma$ we have  $$td_y^T(\O_\sigma\hookrightarrow X)=(1+y)^{\dim \O_\sigma}td^T(O_\sigma)\,.$$
The claim is true for smooth complete toric varieties by Lemma \ref{llemm}, since each
orbit is a SNC divisor complement in its closure. The  general case follows since each toric variety can be completed
and resolved in the category of toric varieties.\end{proof}

The nonequivariant version of Corollary \ref{toric} appeared in
\cite{MaSc}. The  Euler-Maclaurin formula for the Todd class of a
closed orbit was given by many authors, see e.g.  \cite{BV}
(for simplicial cones) and by Brylinski-Zhang \cite{BZ} in general.
One has to sum the lattice points in the dual cones
$$td^\T(id_{X_\Sigma})=\sum_{\sigma\in \Sigma'}\sum_{m\in\sigma^\vee} e^{-m}[\O_\sigma]\,,$$
where the sum is taken over the cones of maximal dimension indexed
by $\Sigma'\subset\Sigma$. The formula makes sense under the
identification of the dual lattice ${\rm Hom}(\T,\C^*)$ with
$H^2_\T(pt)$.

To compute the class $td^\T_y(\O_\tau \hookrightarrow X)$ of an orbit we have to
modify the formula: the closure of the orbit is defined by the fan
which combinatorially is the link of $\tau$. From the Hirzebruch
class of $\overline{\O_\sigma}$ one has to subtract the classes of
the boundary. We obtain:
$$td^\T(\O_\tau\hookrightarrow X)=\sum_{\begin{matrix}\sigma\in \Sigma'\\\sigma\succ
\tau\end{matrix}}\;\sum_{m\in int(\sigma^\vee\cap \tau^\perp)} e^{-m}[\O_\sigma]\,.$$
We sum the contributions coming from orbits to obtain a global formula
$$td^\T_y(id_X)=\sum_{\tau\in \Sigma}(1+y)^{\dim \O_\tau}\sum_{\begin{matrix}\sigma\in \Sigma'\\\sigma\succ
\tau\end{matrix}}\;\sum_{m\in int(\sigma^\vee\cap \tau^\perp)} e^{-m}[\O_\sigma]\,.$$
The local Hirzebruch class is given by a modified Euler-Maclaurin sum:

\begin{theorem}\label{toma}
The restriction of the Hirzebruch class to the fixed point
corresponding to a maximal cone $\sigma\in \Sigma'$ is equal
$$td^\T_y(id_X)){|p_\sigma}=\sum_{\sigma\succ
\tau}(1+y)^{\dim \O_\tau}\sum_{m\in int(\sigma^\vee\cap \tau^\perp)} e^{-m}[\O_\sigma]\,.$$
\end{theorem}

\begin{example}\label{g24}\rm  Suppose $\T=(\C^*)^4$ acts on $\C^4$ with
weights $t_3-t_1\,,\,t_4-t_1\,,\,t_3-t_2\,,\,t_4-t_2$ (in fact a
quotient 3-dimensional torus acts effectively). Let $X$ be defined
by the equation $x_1x_4=x_2x_3$. It is the affine
cone over $\P^1\times\P^1\subset \P^3$.
The variety $X$ is an affine toric variety associated to a cone $\sigma_X\in \Z^3$.
The dual cone of $X$  $$\sigma_X^\vee\subset \{(t_1,t_2,t_3,t_4)\in \Z^4\,:\,t_1+t_2+t_3+t_4=0\}\simeq\Z^3$$ is spanned by the rays
$t_3-t_1\,,\,t_4-t_1\,,\,t_3-t_2\,,\,t_4-t_2$. (This is \cite[Example of \S1.3, p.18]
{Futor} with the notation
$e^*_1=t_3-t_1$,
$e^*_3=t_4-t_1$,
$e^*_1+e^*_2=t_3-t_2$,
$e^*_2+e^*_3=t_4-t_2$.)
We compute Euler-Maclaurin sum for each face of the cone and multiply it by a power of $1+y$. Let
$T_{ij}=e^{t_i-t_j}$. The sums are equal to:
{\def\gg#1{\frac{T_{#1}}{1-T_{#1}}}\begin{itemize}
\item 1 for the vertex.
\item For the rays:

$(1+y)\left(\sum_{i=1}^\infty T_{13}^i+\sum_{i=1}^\infty T_{14}^i+\sum_{i=1}^\infty T_{23}^i+\sum_{i=1}^\infty T_{24}^i\right)=$

\hfill$=(1+y)\left(\gg{13}+\gg{14}+\gg{23}+\gg{24}\right)$.

\item Similarly for the two dimensional faces

$(1+y)^2\left(\gg{13}\gg{14}+\gg{13}\gg{23}+\gg{14}\gg{24}+\gg{23}\gg{24}\right)$.
\item To compute the sum of the entries in the interior we divide the cone into two simplicial cones. The first one is spanned by $t_3-t_1,t_4-t_1, t_3-t_2$, the second by $t_4-t_2,t_4-t_1, t_3-t_2$. The common face is spanned by $t_4-t_1, t_3-t_2$, see \cite[p.49]{Futor}. The rays of these cones generate the corresponding semigroups of integral points, therefore the summand corresponding to the interior of the cone is equal to

$(1+y)^3\left(\gg{13}\gg{14}\gg{23}+\gg{42}\gg{14}\gg{23}+\gg{14}\gg{23}\right)$.
\end{itemize}
}
\noindent Let
$S_{ij}=T_{ij}-1$ and substitute $(y+1)$ by $-\delta$.
 The localized Hirzebruch class of $X$ is
equal to
{\def\ff#1{\frac{1+S_{#1}}{S_{#1}}}
\begin{equation} \label{minus}1 + \delta
\left(\ff{13} + \ff{14} + \ff{23}+ \ff{24}\right)+\end{equation}

$$+ \delta ^2 \left(\ff{13}\ff{14}+ \ff{13}\ff{23} + \ff{14}\ff{24} + \ff{23} \ff{24}\right) +$$

\begin{equation*}  +\delta ^3\left(\ff{13}\ff{14}\ff{23}+\ff{42}\ff{14}\ff{23}-\ff{14}\ff{23}\right)\,.\end{equation*}
After simplification the last summand is equal to
$$\delta ^3 \frac{\left(1+S_{14}\right) \left(1+S_{23}\right) \left(S_
{13} + S_{24} + S_{13} S_{24}\right)}{S_{13} S_{14} S_{23} S_{24}}\,.$$
}
\end{example}

\section{$A_n$ singularities}
We consider here the surface singularities of   type $A_n$ they
are important but easy from the computational point of view. On
one hand they form the first series in the list of simple
singularities, on the other hand they are quotient singularities of
$\C^2$ by a cyclic group contained in $SL_2(\C)$, and finally they admit an action of a
2-dimensional torus and the computations can be done by means to
Theorem \ref{toma}.

Let us define $A_{n-1}$ singularity as the quotient $\C^2$ by
$\Z_n\subset\C^*$ which acts on $\C^2$ by the formula
$\xi(z_1,z_2)=(\xi z_1,\xi^{-1} z_2)$. The ring of invariants is
generated by the monomials $x_1:=z_1^n$, $x_2:=z_2^n$,
$x_3:=z_1z_2$. Another description of this singularity is given by
the zeros of the polynomial $$X=\{(x_1,x_2,x_3)\in
\C^3\,|\,x_1x_2=x_3^n\}$$ in $\C^3$. The torus $\T=(\C^*)^2$
acting on $\C^3$ with weights $w_1=n\,t_1$, $w_2=n\,t_2$,
$w_3=t_1+t_2$ preserves $X$. Let $$H=\{(t_1,t_2)\in \T\;:\;t_1t_2=1\,,\;t_1^n=1\}\simeq\Z/n\,.$$ The quotient torus $\T'=\T/H$ acts on $X$ effectively. The affine toric variety $X$ corresponds to a rational cone $\sigma\subset Hom(\C^*,\T')\otimes \Q)=Hom(\T,\C^*)\otimes\Q$. Its
dual cone  $\sigma^\vee\subset  Hom(\T,\C^*)\otimes\Q$ is spanned by the weights $w_1=n\,t_1$ and $w_1=n\,t_2$. The lattice in $Hom(\T,\C^*)\otimes\Q$ corresponding to the quotient torus is generated by
$n\,t_1$,  $n\,t_2$ and
$t_1+t_2$. According to Theorem \ref{toma} we
sum up
$$\begin{matrix}
1 & 0-\text{dimensional orbit}\hfill\\ \\
(y+1)\sum_{k=1}^\infty e^{-k\, n\, t_1} +(y+1) \sum_{\ell=1}^\infty e^{-\ell\, n\, t_2}&1-\text{dimensional orbits}\hfill\\ \\
(y+1)^2 \sum_{k,\ell}^\infty e^{-k\, n\, t_1+\ell\, n\,
t_2}\;\sum_{i=1}^n e^{-i(t_1+t_2)}&2-\text{dimensional
orbit.}\hfill
\end{matrix}$$

\begin{center}\pgfdeclareimage[height=5cm]{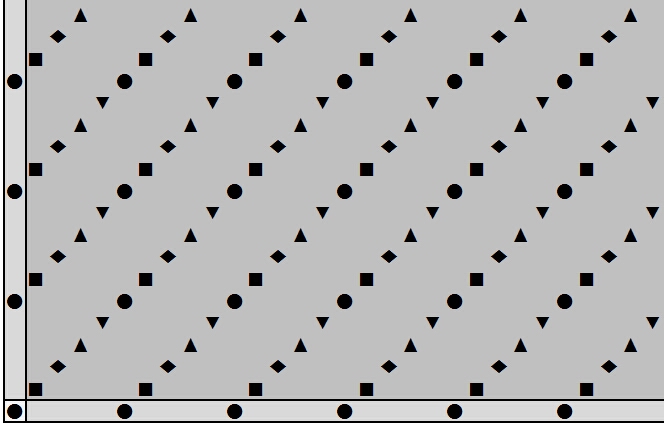}{Anbw}\pgfuseimage{Anbw}
\end{center}

We introduce new
variables $S_{ni}=e^{-n\,t_i}-1$, $S_{12}=e^{-(t_1+t_2)}-1$ and as always
$\delta=-1-y$. Then
$$\frac{td_y^\T(X\hookrightarrow
\C^3)_{|0}}
 {eu(0)}=1+\delta\left(\frac{1+S_{n1}}{S_{n1}}+\frac{1+S_{n2}}{S_{n2}}\right)
 +\delta^2\frac{1}{S_{n1}S_{n2}}\;\sum_{i=1}^n (S_{12}+1)^i
 $$

We see that the class of the $A_{n-1}$ singularity in $\C^3$ is of
the form $\frac1{S_{n1}S_{n2}S_{12}}$ multiplied by a polynomial
in $\delta$, $S_{n1}$, $ S_{n2}$ and $ S_{12}$ with nonnegative
coefficients.
One can check that in our case the local Hirzebruch class of
the complement is  positive as well. Divided by the Euler class it
is equal to
$$
\delta \frac{1+S_{12}}{S_{12}} + \delta^2 \left(
\frac{1+S_{n1}}{S_{n1}}\frac{1+S_{n2}}{S_{n2}}+
\frac{1+S_{12}}{S_{12}} \right) + \delta^3 \frac{1+S_{n1}}{S_{n1}}
\frac{1+S_{n2}}{S_{n2}}\frac{1+S_{12}}{S_{12}}
$$

The remaining du Val singularities, $D_n$ series and $E_n$ singularities are
 invariant with respect to actions of one dimensional torus.  By a direct computation one finds that
 the Hirzebruch classes of these singularities are positive in variables $\delta=-1-y$ and $S=T-1$.

\section{Positivity for simplicial toric varieties}

The method of summation of lattice points applied in the case of $A_n$ can be  generalized\footnote{I thank Oleg Karpenkov for driving my attention to this method of summation}.
\begin{theorem}\label{torsimp}Let $X$ be the toric variety defined by a   simplicial cone $\sigma\subset Hom(\C^*,\T)$. We fix an equivariant embedding $X\to\C^N$ given by a choice of generators $$w_1,w_2,\dots, w_N\in Hom(\T,\C^*)\cap \sigma^\vee\,.$$
We assume that $w_1,w_2,\dots, w_n$ are the primitive vectors spanning the rays of $\sigma^\vee$.
Then the Hirzebruch class of the open orbit $td_y^\T(\O\hookrightarrow\C^N)$ is of the form
$$\delta ^n \prod_{i=1}^r\frac1S_{w_i}\cdot P(\{S_{w_i}\}_{i=1,2,\dots,N})\,,$$ where $ P(\{S_{w_i}\}_{i=1,2,\dots,N})$ is a polynomial expression with nonnegative coefficients depending on the variables $S_{w_i}=e^{-w_i}-1$.\end{theorem}

\begin{proof}
 The vectors $w_1,w_2,\dots, w_n$ generate a sublattice  $\Lambda\subset Hom(\T,\C^*)$. The quotient group $(Hom(\T,\C^*)/\Lambda$ is finite. Let $A$ be  the closed cube spanned by $w_1,w_2,\dots, w_n$ and
 $$A_0=A\cap int(\sigma^\vee)\cap Hom(\T,\C^*)\,.$$
 Then $$Hom(\T,\C^*)\cap int(\sigma^\vee)=\bigsqcup_{w\in A_0} \; \bigsqcup_{k_1,k_2,\dots,k_{n}>0} \left\{w+\sum_{i=1}^{n}k_jw_i\right\}\,.$$
The Euler-Maclaurin sum is equal to
$$\sum_{w\in  A_0} e^{-w}\prod_{i=1}^{n}\frac{1}{1-e^{-w_{i}}}=(-1)^n\sum_{w\in  A_0}(S_w+1)\prod_{j=1}^{n_i}\frac { 1}{S_{w_{i }}}\,.$$
Since $S_{w+w'}=S_wS_{w'}+S_w+S_{w'}$, the above formula can be rewritten in terms of the generators of $Hom(\T,\C^*)\cap\sigma^\vee$. Applying Theorem \ref{toma}
we arrive to the conclusion. Note that since $y+1=-\delta$ the signs cancel.\end{proof}

If $\sigma$ is not simplicial, then we can apply the same method of summation. The dual cone $\sigma^\vee$ can be divided into a finite number of open simplicial  cones. Some of these the cones will be of lower dimension. Multiplying by $(-\delta)^n$ we do not get rid of sign alternation (see the minus in Example \ref{g24}, formula (\ref{minus})). Indeed in the dimension four there exist examples of toric singularities with Hirzebruch class, which is not positive. One can check that the cone over the suspension of a pentagon (precisely the cone spanned by
$P_1 = (0, 0, 1, 1),\, P_2 = (1, 0, 1, 1),\, P_3 = (2, 1, 1, 1),\, P_4 = (1, 2,
  1, 1),\, P_5 = (0, 1, 1, 1),\,R_1 = (1, 1, 2, 1)$ and $R_2 = (1, 1, 0, 1)$)
has non-positive Hirzebruch class.

Instead in the dimension three we have:

\begin{proposition}Let $X$ be the toric variety defined by a  three dimensional  cone $\sigma\subset Hom(\C^*,\T)$.
Suppose that the dual cone is spanned by the primitive vectors $w_1,w_2,\dots,w_k$ and $w_{k+1},w_{k+2},\dots, w_N$ are remaining generators of $\sigma^\vee$.
Then the Hirzebruch class of the open orbit $td_y^\T(\O\hookrightarrow\C^N)$ is of the form
$$\delta ^3 \prod_{i=1}^k\frac1S_{w_i}\cdot P(\{S_{w_i}\}_{i=1,2,\dots,N})\,,$$ where $ P(\{S_{w_i}\}_{i=1,2,\dots,N})$ is a polynomial expression with nonnegative coefficients depending on the variables $S_{w_i}=e^{-w_i}-1$.\end{proposition}

\begin{proof} The cone $\sigma^\vee$ is the cone over a  $k$-polygon $P$.
This polygon can be
divided  into triangles in a way
that no additional vertex is introduced. Intersection of two triangles is an edge.
 Consequently, the polygon $P$ is the disjoint union $int(P)=\bigsqcup_{i=1}^{k-2}P_i$, where $P_1$ is an open triangle and $P_i$ for $i>1$ are triangles with one edge added. Modifying the proof of Theorem
\ref{torsimp} we show that each piece of decomposition gives a nonnegative contribution to $td^\T_y(X)$. Precisely, for the cones $\tau_i=cone(P_i)$,
 $i>1$ we choose $A_0$ to contain the corresponding edge monomials of $\tau_i$, not its opposite in the interior.\end{proof}

\section{Comparison with Baum-Fulton-MacPherson class}
\label{bfm} The equality \begin{equation}td^\T_0(X\hookrightarrow
\C^n)=td^\T(\C^n)ch^\T(\O_X)\label{rownosc}\end{equation} does not hold always. The simplest
example is the cusp of the type $A_{2n}$.
\begin{example}\rm Let $X=\{x^2=y^{2n+1}\}\subset \C^2$ with the torus
$\C^*$ action having weights $(2n+1)t\,,\, 2t$. Then
$$td^\T_0(X\hookrightarrow
\C^2)=f_*(td^\T(\C))\,,$$ where $f:\C\to\C^2\,,$
$f(z)=(z^{2n+1},z^2)$ is the normalization and $\C$ is equipped
with the action of $\C^*$ of weight $t$. The fundamental class of $X$ in
$H^*_\T(\C^2)$ is equal to $2(2n+1)t$. Hence
$$td^\T_0(X\hookrightarrow
\C^2)=2(2n+1)t\frac t{1-T}\,,$$
while Baum-Fulton-MacPherson class is equal to
$$td(\C^2)ch^\T\O_X=td(\C^2)ch^\T\big(\O_{\C^2}-\O_{\C^2}(-X)\big)=\frac{ (2n+1)t}{1-T^{2n+1}} \frac
{2t}{1-T^2}\big(1-T^{2(2n+1)}\big)\,.$$
\end{example}
Also taking the affine cone over a smooth hypersurface in $\P^{n-1}$ of
degree $d>n$ (applying Proposition \ref{prco}) we  obtain
counterexamples which are normal. We treat the equality (\ref{rownosc}) as a
special property of the singularity germ. On the level of
$K$-theory we simply ask if
$$mC_0(id_X)=[\O_X]\,.$$
If it is the case then Baum-Fulton-MacPherson-Todd  class \cite{BFM} of $X$
  coincides with
$td_0(id_X)$. The same holds in the equivariant setup.

Let us recall some conditions for the equality. There exists a
characterization of the Du Bois singularities which fits to our
situation the best. It is given in \cite{Sc}. Suppose that $X$ is
a subvariety in a smooth ambient space, $f:\widetilde{M}\to M$ is a proper map such that $f_{|f^{-1}(M\setminus X)}$ is an
isomorphism and $D=f^{-1}(X)$ is a smooth divisor with normal
crossings (we say $f$ is a resolution of the pair $(M,X)$). Then $X$ has Du Bois singularities if and only if the
natural map
$$\O_X\to Rf_*\O_D$$ is a quasi-isomorphism.  Then
$$mC_0(M\setminus X\hookrightarrow
M)=f_*mC_0(\widetilde{M}\setminus X\hookrightarrow
\widetilde{M})=[Rf_*\O_{\widetilde{M}}]-[Rf_*\O_D]=[\O_{M}]-[\O_X]\,.$$
Hence
$$mC_0(X\hookrightarrow
M)=[\O_X]\,.$$
 (Compare \cite[Example 3.2]{BSY}.) If $X$ has rational
singularities i.e.~$X$ is normal and
$Rf_*\O_{\widetilde{X}}=f_*\O_{\widetilde{X}}=\O_X$, then by
\cite{Ko} it has at most Du Bois singularities. Most of the examples we
have considered here have rational singularities:  Schubert varieties by
\cite[Theorem 4]{Ra}, toric varieties by \cite[Cor. 3.9]{Od}. One can add to this list the cones over a smooth hyperplane in $\P^{n-1}$ (considered in Corollary \ref{conicalra}) provided that the degree is smaller than $n$. (If $d=n$ the equality still holds, but the singularity is not rational.) If
$X$ is a surface with rational singularities, then the exceptional
divisor of $X$ is a tree of rational curves. If
$f:\widetilde{X}\to X$ is a resolution of rational singularities
of higher dimension, then the associated complex of intersections
of divisors is $\Q$-acyclic, \cite[Theorem 3.1]{ABW}.
 But
only from rationality we do not get enough information about
derived direct images of $L\Omega^y_{D_I}$. There should exist a
natural class of varieties for which the positivity discussed
below holds.

\section{Question of positivity}
\label{udod}
Let $\T=(\C^*)^r$ be a torus acting on $\C^n$ with weights which
are nonnegative combinations of the basis characters. Suppose that
$X\subset \C^n$ is an invariant subvariety. As it was explained in
the introduction we develop the class
$$mC_*(X\hookrightarrow \C^n)\in
G_{\T}(\C^n)[y]=K_{\T}(\C^n)[y]=K_{\T}(pt)[y]$$ in the basis of
monomials in $\delta=-1-y$ and $S_i=T^i-1$. Let us assume that $X$
is an open subset and let $f:\widetilde{\C}^n\to \C^n$ be a
resolution of the pair $(\C^n,X)$. Then using
the notation of \S\ref{divisor}
$$mC_*(X)=\sum_I \delta ^{|I|}Rf_*(L\Omega^{y}_I)\,.$$
Examples considered by us (quadratic cones of Proposition \ref{kwadr}, simplicial toric varieties) show that all sheaves (or their
Chern characters)
$$Rf_*(L\Omega^{y}_I)=\sum_{p=0}^{\dim X-|I|}(-1-\delta)^p
Rf_*(L\Omega^{p}_I)$$
 have nonnegative coefficients in the
expansion. On the other hand further decomposition into sheaves
$(-1-\delta)^p Rf_*(L\Omega^{p}_I)$ does not preserve positivity. A
counterexample is the cone over a smooth quadric. It suggests that
the sheaf $Rf_*L\Omega^y_I=f_*\Lambda_{-1-\delta}(T^*X_I(-\log D ))$
should be treated as a single object.

We  would like to point out that the positivity of the class
$\alpha\in K_\T(\C^n)$ can be understood in various ways
\begin{enumerate}
\item $\alpha$ is represented by an effective sheaf
\item $\alpha$ is represented by  a sum of the sheaves
$(-1)^{\codim Y}[\O_Y]$ where $Y$ are subvarieties in $\C^n$
\item $\alpha$ is  represented by a sum of the sheaves
$(-1)^{\codim Y}[\O_Y]$ where $Y$ are subvarieties in $\C^n$ with
rational singularities.
\item $\alpha$ is a polynomial in
$S_w$  with nonnegative coefficients, where $w$ are
the weights of the action of $\T$ on $\C^n$.
\end{enumerate}
We note that
\begin{theorem} The condition 3) implies
that $\alpha$ is  a polynomial in $S_{t_i}$  with
nonnegative coefficients, where $t_i$ is the positive basis of characters
fixed  in the beginning.\label{posa}\end{theorem}
\begin{proof}
Let $Y\subset \C^n$ be an invariant subvariety with rational
singularities. Fix an approximation of the classifying space
$B\T_m=(\P^m)^r$ with the universal $\T$-bundle
$E\T_m=\boxtimes_{i=1}^r\O^*(1)$. Then the associated bundle
obtained via Borel construction $ET_\m\times_\T \C^n$ is
globally generated. Let $s$ be a generic section. Then
$s^{-1}(ET_\m\times_\T Y)$ has expected dimension, and it has at
most rational singularities. Let $X_I\subset (\P^m)^r$ be the
closure of a cell of the standard decomposition of the product of
projective spaces. By our choice of variables
$[\O_{X_I}]=(-1)^{\codim X_I}S_I$ under the identification of
$K_\T(\C^n)$ with $K_\T(pt)$ given by the section $s$. Applying
\cite[Theorem 1]{Br} for the homogeneous space $B\T_m=(\P^m)^r$ we
find that
$$[\O_{ET_\m\times_\T Y}]=\sum(-1)^{\codim Y-\codim X_I}c_I[\O_{X_I}]\,,$$
where  the numbers $c_I$ are nonnegative.
 \end{proof}

\section{Schubert varieties and cells}\label{dod}
The singularities of Schubert varieties in flag manifolds $G/B$ or more general in homogeneous spaces $G/P$ were studied by many authors, see for example the monograph \cite{BiLa}, especially \S4.4, or \cite{WoYo}. The maximal torus $T\subset G$ has a discrete fixed point set in $G/P$. Each fixed point corresponds to a Schubert cell. Whenever a Schubert cell contains another cell in the closure, one can ask what is the singularity at the boundary point. Even the question of smoothness has turned out to be difficult, see \cite[\S4.6]{BiLa}. Localized Hirzebruch class of the boundary point is an invariant which describes in a way the singularity. For example if the point is smooth, then the local Hirzebruch class decomposes into linear factors.  We have  computed localized Hirzebruch classes in some cases and noticed certain phenomenon.
Various computation experiments show that   the class
 $$mC_*(X\hookrightarrow
\C^n))_{|p}=[f_{\bullet*}\Omega^y_{X_\bullet}]_{|p}$$
 is positive for
Schubert varieties. The Schubert varieties are sums of open cells and
the positivity holds for cells as well. The expansion should be in
variables $\delta$ and $S_i$ for a  positive basis
 in the sense of \cite{AGM}. We have already given an example of the
cone over a quadric in $\P^1$, Example \ref{stozLGe} formula (\ref{stozLG}). The Hirzebruch class written in variables $\delta$ and $S_w=e^{-w}-1$ is equal to
$$\frac{ S_{2t_1}+ S_{2t_2} + \delta (2 S_{t_1+t_2} + S_{t_1+t_2}^2 + S_{2t_1} S_{2t_2}) + \delta^2 (S_{t_1+t_2} + 1) (S_{t_1+t_2} + 2)}{S_{2t_1} S_{2t_2}}\,.$$
 This is the singularity of the
codimension one Schubert variety in the Lagrangian Grassmannian
$LG(2)$. The class of the complement is equal to:{\small
$$\frac{\delta (S_{t_1+t_2} + 1) S_{2t_1} S_{2t_2} + \delta^2 (S_{t_1+t_2} + 1)
(S_{t_1+t_2}^2 + S_{2t_1} S_{2t_2} + S_{t_1+t_2}) + \delta^3 (S_{t_1+t_2} +
1)^3}{S_{2t_1} S_{2t_2}S_{t_1+t_2}}\,.$$ } It is remarkable that the
Hirzebruch class can be expressed by the variables associated to
the characters of the representation, although they do not span
whole equivariant cohomology of a point. This phenomenon is easy
to explain: resolving the singularity of $X$ we construct new
$\T$-varieties for which the weights of the tangent spaces are
combinations of the original weights.

The method of \cite{We} allows to compute the Hirzebruch class of
the codimension one Schubert variety in the classical Grassmannian
$Gr_n(\C^{2n})$. Locally this Schubert variety  is equal to the determinantal variety consisting
of singular quadratic matrices. The case $n=2$ is the quadratic
singularity in $\C^4$, which is toric. The Hirzebruch class was
computed in Example \ref{g24}, and the formula for the open complement is
the following:
$$\frac{\delta^2 (1 + S_{13}) (1 + S_{24}) }{S_{13}S_{14}S_{23}S_{24}}\times
\phantom{\Big(S_{14} S_{23} + S_{13} S_{24}+   \delta (S_{14} + S_{23} + 2 S_{14} S_{23} + S_{13} S_{24}))
 + \delta^2 (1 + S_{13}) (1 + S_{24})\Big)}$$
 \hfill$\times\Big(S_{14} S_{23} + S_{13} S_{24}+   \delta (S_{14} + S_{23} + 2 S_{14} S_{23} + S_{13} S_{24}))
 + \delta^2 (1 + S_{13}) (1 + S_{24})\Big)\,.$\vskip10pt

\noindent The formula is not symmetric since there is a relation between variables
$$(1 + S_{13}) (1 + S_{24})=(1 + S_{23}) (1 + S_{14})\,.$$
For bigger $n$ the formulas are quite complicated. Below we give
the result for $n=3$ restricted to one dimensional torus acting diagonally.
$$\frac{(\delta+1)^3(S+1)^3}{S^9}\Big(
6 S^6+
9\delta S^5 (2+3 S)+
5\delta^2 S^4 (6+15 S+10 S^2)+$$
$$+\delta^3S^3 (30+105 S+123 S^2+49 S^3)
+9\delta^4 S^2 (1+S)^3 (2+3 S)+$$
$$+\delta^5S (1+S)^3 (6+15 S+8 S^2)+
\delta^6(1+S)^6\Big)$$

Let us give another example.

\begin{example}\rm Let  $X$ be the open cell in the flag variety $Fl(3)=GL_3(\C)/B$.
The 0-dimensional Schubert cell which corresponds to the standard
flag has a neighbourhood (the opposite open cell) which can be
identified with the set of lower-triangular matrices with 1's at
the diagonal
$$\left(\begin{matrix}1& 0&0\\x_{21}&1&0\\x_{31}&x_{32}&1\end{matrix}\right)\,.$$
The three dimensional diagonal torus is acting   on $Fl(3)$
preserving the cell decomposition and the decomposition into opposite cells.  The intersection of the open cell with the opposite cell is  the complement of two
divisors $D_1=\{x_{31}=0\}$ and $D_2=\{x_{21}x_{32}-x_{31}=0\}$
defined by vanishing of lower-left corner determinants.
\begin{center}\pgfdeclareimage[height=6cm]{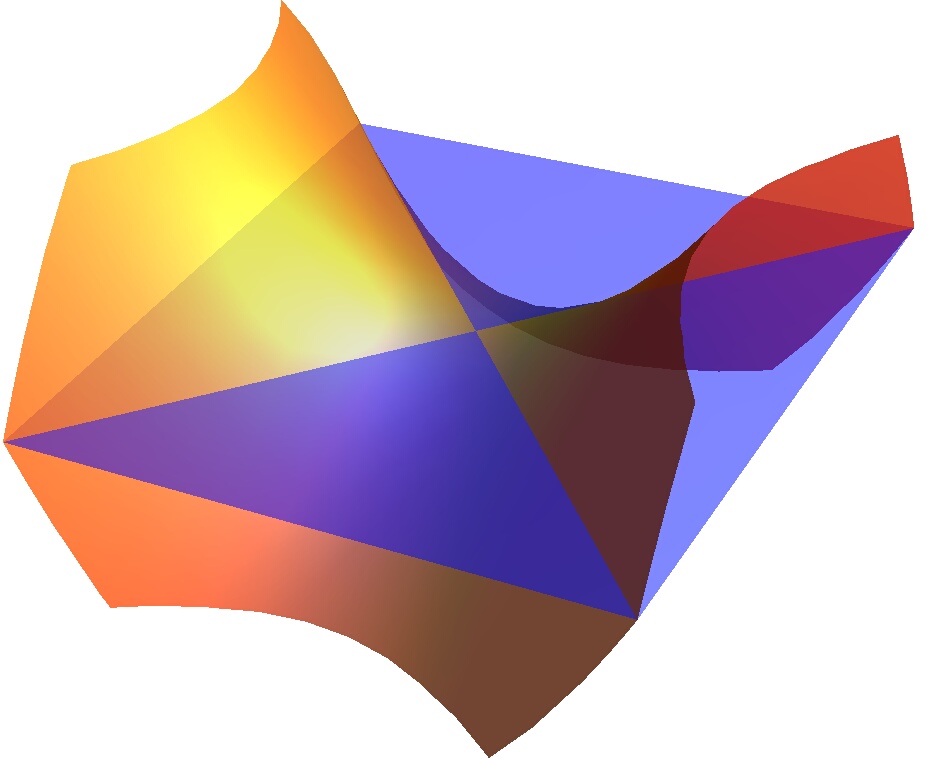}{fl3}\pgfuseimage{fl3}
\end{center}
The weight of the variable
$x_{ij}$ is $t_i-t_j$. The local
Hirzebruch class of the
open cell
is equal to
\begin{align*}&td_y^\T(\C^3)_{|0}-td_y^\T(D_1\hookrightarrow
\C^3)_{|0}-td_y^\T(D_2\hookrightarrow \C^3)_{|0}+td_y^\T(D_1\cap
D_2 \hookrightarrow \C^3)_{|0}\,.\end{align*}
 The intersection of the divisors is not transverse, but it is equal
 to the intersection of two coordinate lines $x_{21}=0$ and $x_{32}=0$ in the plane
 $\{x_{31}=0\}$. Therefore by the inclusion-exclusion formula
$td_y^\T(X\hookrightarrow \C^3)_{|0}$ is equal to
{\def\ff#1#2{\frac{1+y\frac{T_#2}{T_{#1}}}{1-\frac{T_#2}{T_{#1}}}}
\begin{align*}
\ff 1 2\cdot\ff 1 3\cdot\ff 2 3&-2\ff 1 2\cdot\ff 2 3
+\ff 1 2+\ff 2 3-1\,,
\\
\end{align*}}
where $T_i=e^{-t_i}$. After the substitution $S_{ij}=T_i/T_j-1$ and using the relation
$(1 + S_{21}) (1 + S_{ 32})=1 + S_{ 31}$ the Hirzebruch class can be written in the form
$$\frac\delta{S_{21}S_{31}S_{32}}(1+S_{21})
(1+S_{31})\big(S_{21 } S_{32} +\delta S_{21 } S_{32 } + \delta^2 (1 + S_{31 })\big)
\,,$$
\end{example}

We ask: Does  positivity always holds for Schubert cells in
a homogeneous space $G/P$? After substitution $T_i=e^{\delta
t_i}$ and letting $\delta\to 0$ we obtain
Chern-Schwartz-MacPherson class. The positivity of the {\it nonequivariant} Chern-Schwartz-MacPherson class in
the case of ordinary Grassmannians  (conjectured in \cite{AlMi}) was shown by Huh \cite{Huh}.
 Proving nonequivariant result Huh takes an advantage of the action of the Borel $B$ group preserving the cells.
The proof is based on the observation that the bundle
$T\widetilde{M}(-\log D)$ has a lot sections for a suitable
$B$-invariant resolution of the Schubert cell.
In general Huh considers the situation when a solvable group $B$ acts on the pair
$(M,U)$, where $U$ is a an open set contained in the nonsingular locus of $M$. Under the
assumption that the pair $(M,U)$ admits a $B$-equivariant resolution
$(\widetilde M, \widetilde U)$ on which $B$ acts with finitely many orbits it is shown
that the Chern-Schwartz-MacPherson class is effective. The equivariant version of this
result would give positivity after specialization $\delta=0$. Nevertheless the example of
non-simplicial toric varieties shows that the assumption of Huh is not strong enough to imply
positivity of the Hirzebruch class.


\begin{thebibliography}{GKM98}

\bibitem[AB84]{AB}
Michael~F. {Atiyah} and Raoul {Bott}.
\newblock {The moment map and equivariant cohomology.}
\newblock {\em {Topology}}, 23:1--28, 1984.

\bibitem[ABW13]{ABW}
Donu Arapura, Parsa Bakhtary, and Jaros{\l}aw W{\l}odarczyk.
\newblock Weights on cohomology, invariants of singularities, and dual
  complexes.
\newblock {\em Math. Ann.}, 357(2):513--550, 2013.

\bibitem[AFP14]{AFP}
Christopher Allday, Matthias Franz, and Volker Puppe.
\newblock Equivariant cohomology, syzygies and orbit structure.
\newblock {\em Trans. Amer. Math. Soc.}, 366(12):6567--6589, 2014.

\bibitem[AGM11]{AGM}
Dave Anderson, Stephen Griffeth, and Ezra Miller.
\newblock Positivity and {K}leiman transversality in equivariant {$K$}-theory
  of homogeneous spaces.
\newblock {\em J. Eur. Math. Soc. (JEMS)}, 13(1):57--84, 2011.

\bibitem[Alu99]{Alog}
Paolo Aluffi.
\newblock Differential forms with logarithmic poles and
  {C}hern-{S}chwartz-{M}ac{P}herson classes of singular varieties.
\newblock {\em C. R. Acad. Sci. Paris S\'er. I Math.}, 329(7):619--624, 1999.

\bibitem[AM09]{AlMi}
Paolo Aluffi and Leonardo~Constantin Mihalcea.
\newblock Chern classes of {S}chubert cells and varieties.
\newblock {\em J. Algebraic Geom.}, 18(1):63--100, 2009.

\bibitem[AM11]{AMfe}
Paolo Aluffi and Matilde Marcolli.
\newblock Algebro-geometric {F}eynman rules.
\newblock {\em Int. J. Geom. Methods Mod. Phys.}, 8(1):203--237, 2011.

\bibitem[AS69]{AS}
M.~F. Atiyah and G.~B. Segal.
\newblock Equivariant {$K$}-theory and completion.
\newblock {\em J. Differential Geometry}, 3:1--18, 1969.

\bibitem[Bau82]{Ba}
Paul Baum.
\newblock Fixed point formula for singular varieties.
\newblock In {\em Current trends in algebraic topology, {P}art 2 ({L}ondon,
  {O}nt., 1981)}, volume~2 of {\em CMS Conf. Proc.}, pages 3--22. Amer. Math.
  Soc., Providence, R.I., 1982.

\bibitem[BB73]{B-B}
A.~Bia{\l}ynicki-Birula.
\newblock Some theorems on actions of algebraic groups.
\newblock {\em Ann. of Math. (2)}, 98:480--497, 1973.

\bibitem[BFM75]{BFM}
Paul Baum, William Fulton, and Robert MacPherson.
\newblock Riemann-{R}och for singular varieties.
\newblock {\em Inst. Hautes \'Etudes Sci. Publ. Math.}, (45):101--145, 1975.

\bibitem[BL94]{BL}
Joseph Bernstein and Valery Lunts.
\newblock {\em Equivariant sheaves and functors}, volume 1578 of {\em Lecture
  Notes in Mathematics}.
\newblock Springer-Verlag, Berlin, 1994.

\bibitem[BL00]{BiLa}
Sara Billey and Venkatramani Lakshmibai.
\newblock {\em Singular loci of {S}chubert varieties}, volume 182 of {\em
  Progress in Mathematics}.
\newblock Birkh\"auser Boston, Inc., Boston, MA, 2000.

\bibitem[Bor60]{Bo}
Armand Borel.
\newblock {\em Seminar on transformation groups}.
\newblock With contributions by G. Bredon, E. E. Floyd, D. Montgomery, R.
  Palais. Annals of Mathematics Studies, No. 46. Princeton University Press,
  Princeton, N.J., 1960.

\bibitem[Bri02]{Br}
Michel Brion.
\newblock Positivity in the {G}rothendieck group of complex flag varieties.
\newblock {\em J. Algebra}, 258(1):137--159, 2002.
\newblock Special issue in celebration of Claudio Procesi's 60th birthday.

\bibitem[BSY10]{BSY}
Jean-Paul Brasselet, J{\"o}rg Sch{\"u}rmann, and Shoji Yokura.
\newblock Hirzebruch classes and motivic {C}hern classes for singular spaces.
\newblock {\em J. Topol. Anal.}, 2(1):1--55, 2010.

\bibitem[BV82]{BV}
Nicole Berline and Mich{\`e}le Vergne.
\newblock Classes caract\'eristiques \'equivariantes. {F}ormule de localisation
  en cohomologie \'equivariante.
\newblock {\em C. R. Acad. Sci. Paris S\'er. I Math.}, 295(9):539--541, 1982.

\bibitem[BV97]{BrVe}
Michel Brion and Mich{\`e}le Vergne.
\newblock An equivariant {R}iemann-{R}och theorem for complete, simplicial
  toric varieties.
\newblock {\em J. Reine Angew. Math.}, 482:67--92, 1997.

\bibitem[BZ03]{BZ}
Jean-Luc Brylinski and Bin Zhang.
\newblock Equivariant todd classes for toric varieties.
\newblock {\em preprint}, arXiv:math/0311318, 2003.

\bibitem[Del71]{De}
Pierre Deligne.
\newblock Th\'eorie de {H}odge. {II}.
\newblock {\em Inst. Hautes \'Etudes Sci. Publ. Math.}, (40):5--57, 1971.

\bibitem[EG98]{EdGr}
Dan Edidin and William Graham.
\newblock Localization in equivariant intersection theory and the {B}ott
  residue formula.
\newblock {\em Amer. J. Math.}, 120(3):619--636, 1998.

\bibitem[FJ80]{FJ}
William Fulton and Kent Johnson.
\newblock Canonical classes on singular varieties.
\newblock {\em Manuscripta Math.}, 32(3-4):381--389, 1980.

\bibitem[Ful93]{Futor}
William Fulton.
\newblock {\em Introduction to toric varieties}, volume 131 of {\em Annals of
  Mathematics Studies}.
\newblock Princeton University Press, Princeton, NJ, 1993.

\bibitem[Ful98]{Fu}
William Fulton.
\newblock {\em Intersection theory}, volume~2 of {\em Ergebnisse der Mathematik
  und ihrer Grenzgebiete. 3. Folge. A Series of Modern Surveys in Mathematics}.
\newblock Springer-Verlag, Berlin, second edition, 1998.

\bibitem[GKM98]{GKM}
Mark Goresky, Robert Kottwitz, and Robert MacPherson.
\newblock Equivariant cohomology, {K}oszul duality, and the localization
  theorem.
\newblock {\em Invent. Math.}, 131(1):25--83, 1998.

\bibitem[GNA02]{GNA}
Francisco Guill{\'e}n and Vicente Navarro~Aznar.
\newblock Un crit\`ere d'extension des foncteurs d\'efinis sur les sch\'emas
  lisses.
\newblock {\em Publ. Math. Inst. Hautes \'Etudes Sci.}, (95):1--91, 2002.

\bibitem[{Gro}77]{Gr}
A.~{Grothendieck}.
\newblock {Formule de Lefschetz. (Redige par L. Illusie).}
\newblock {Semin. Geom. algebr. Bois-Marie 1965-66, SGA 5, Lect. Notes Math.
  589, Expose No.III, 73-137 (1977).}, 1977.

\bibitem[Hir56]{Hi}
F.~Hirzebruch.
\newblock {\em Neue topologische {M}ethoden in der algebraischen {G}eometrie}.
\newblock Ergebnisse der Mathematik und ihrer Grenzgebiete (N.F.), Heft 9.
  Springer-Verlag, Berlin-G\"ottingen-Heidelberg, 1956.

\bibitem[Huh13]{Huh}
June Huh.
\newblock Positivity of chern classes of schubert cells and varieties.
\newblock {\em preprint}, arXiv:1302.5852, 2013.

\bibitem[Kov99]{Ko}
S{\'a}ndor~J. Kov{\'a}cs.
\newblock Rational, log canonical, {D}u {B}ois singularities: on the
  conjectures of {K}oll\'ar and {S}teenbrink.
\newblock {\em Compositio Math.}, 118(2):123--133, 1999.

\bibitem[MS14]{MaSc}
Lauren{\c{t}}iu Maxim and J{\"o}rg Sch{\"u}rmann.
\newblock Characteristic classes of singular toric varieties.
\newblock {\em Comm. Pure Appl. Math.}, to appear(arXiv:1303.4454), 2014.

\bibitem[MSS13]{MSS}
Laurentiu Maxim, Morihiko Saito, and J{\"o}rg Sch{\"u}rmann.
\newblock Hirzebruch-{M}ilnor classes of complete intersections.
\newblock {\em Adv. Math.}, 241:220--245, 2013.

\bibitem[Mus11]{Mu}
Oleg~R. Musin.
\newblock On rigid {H}irzebruch genera.
\newblock {\em Mosc. Math. J.}, 11(1):139--147, 182, 2011.

\bibitem[MW15]{MiWe}
Ma{\l}gorzata {Mikosz} and Andrzej {Weber}.
\newblock {Equivariant Hirzebruch class for quadratic cones via degenerations.}
\newblock {\em {J. Singul.}}, 12:131--140, 2015.

\bibitem[Oda88]{Od}
Tadao Oda.
\newblock {\em Convex bodies and algebraic geometry}, volume~15 of {\em
  Ergebnisse der Mathematik und ihrer Grenzgebiete (3) [Results in Mathematics
  and Related Areas (3)]}.
\newblock Springer-Verlag, Berlin, 1988.
\newblock An introduction to the theory of toric varieties, Translated from the
  Japanese.

\bibitem[Ohm06]{Oh}
Toru Ohmoto.
\newblock Equivariant {C}hern classes of singular algebraic varieties with
  group actions.
\newblock {\em Math. Proc. Cambridge Philos. Soc.}, 140(1):115--134, 2006.

\bibitem[Par88]{Pa}
Adam Parusi{\'n}ski.
\newblock A generalization of the {M}ilnor number.
\newblock {\em Math. Ann.}, 281(2):247--254, 1988.

\bibitem[PW08]{PW}
Piotr Pragacz and Andrzej Weber.
\newblock Thom polynomials of invariant cones, {S}chur functions and
  positivity.
\newblock In {\em Algebraic cycles, sheaves, shtukas, and moduli}, Trends
  Math., pages 117--129. Birkh\"auser, Basel, 2008.

\bibitem[Qui71]{Qu}
Daniel Quillen.
\newblock The spectrum of an equivariant cohomology ring.i.
\newblock {\em Ann. of Math. (2)}, 94:549--572, 1971.

\bibitem[Ram85]{Ra}
A.~Ramanathan.
\newblock Schubert varieties are arithmetically {C}ohen-{M}acaulay.
\newblock {\em Invent. Math.}, 80(2):283--294, 1985.

\bibitem[Sch07]{Sc}
Karl Schwede.
\newblock A simple characterization of {D}u {B}ois singularities.
\newblock {\em Compos. Math.}, 143(4):813--828, 2007.

\bibitem[Sch11]{Scu}
J{\"o}rg Sch{\"u}rmann.
\newblock Characteristic classes of mixed {H}odge modules.
\newblock In {\em Topology of stratified spaces}, volume~58 of {\em Math. Sci.
  Res. Inst. Publ.}, pages 419--470. Cambridge Univ. Press, Cambridge, 2011.

\bibitem[Seg68]{Se}
Graeme Segal.
\newblock Equivariant {$K$}-theory.
\newblock {\em Inst. Hautes \'Etudes Sci. Publ. Math.}, (34):129--151, 1968.

\bibitem[Tot07]{To}
Burt Totaro.
\newblock The elliptic genus of a singular variety.
\newblock In {\em Elliptic cohomology}, volume 342 of {\em London Math. Soc.
  Lecture Note Ser.}, pages 360--364. Cambridge Univ. Press, Cambridge, 2007.

\bibitem[Web12]{We}
Andrzej Weber.
\newblock Equivariant {C}hern classes and localization theorem.
\newblock {\em J. Singul.}, 5:153--176, 2012.

\bibitem[Web13]{We2}
Andrzej Weber.
\newblock Computing equivariant characteristic classes of singular varieties.
\newblock {\em RIMS Kokyuroku}, 1868:109--129, 2013.

\bibitem[Web14]{Wbb}
Andrzej Weber.
\newblock Hirzebruch class and {B}ia{\l}ynicki-{B}irula decomposition.
\newblock {\em preprint}, arXiv:1411.6594, 2014.

\bibitem[WY08]{WoYo}
Alexander Woo and Alexander Yong.
\newblock Governing singularities of {S}chubert varieties.
\newblock {\em J. Algebra}, 320(2):495--520, 2008.

\bibitem[Yok94]{Yo}
Shoji Yokura.
\newblock A generalized {G}rothendieck-{R}iemann-{R}och theorem for
  {H}irzebruch's {$\chi_y$}-characteristic and {$T_y$}-characteristic.
\newblock {\em Publ. Res. Inst. Math. Sci.}, 30(4):603--610, 1994.

\end{thebibliography}
\end{document}